\documentclass[12pt]{amsart}
\usepackage{amssymb,amsmath,amsfonts,amsthm,comment,mathrsfs,times,graphicx}

\usepackage{tikz}
\usetikzlibrary{arrows,shapes}
\usetikzlibrary{fadings}
\usetikzlibrary{intersections}
\usetikzlibrary{decorations.pathreplacing}

\DeclareMathOperator{\re}{\mathbb{R}e}
\DeclareMathOperator{\im}{\mathbb{I}m}

\newcommand{\INF}{{\infty}}
\newcommand{\eps}{\epsilon}

\newcommand{\dist}{\mbox{dist}}
\newcommand{\tta}{\theta}

\newcommand{\nn}{\nonumber}
\newcommand{\OM}{\Omega}

\newcommand{\sph}{{{\mathbf S}^ 1}}

\newcommand{\del}{\partial}

\newcommand{\Gam}{\varGamma}
\newcommand{\ol}{\overline}
\newcommand{\ds}{\displaystyle}
\newcommand{\dba}{\overline{\partial}}

\newcommand{\fii}{{\varphi}}
\newcommand{\bu}{{\bf u}}
\newcommand{\bv}{{\bf v}}

\newcommand{\bg}{{\bf g}}
\newcommand{\bn}{{\bf n}}

\newcommand{\HT}{\mathcal{H}}
\newcommand{\HAT}{\mathcal{H}_{a}}

\newcommand{\Pminus}{\mathcal{P}^{-}}
\newcommand{\Pplus}{\mathcal{P}^{+}}
\newcommand{\Pstar}{\mathcal{P}^{*}}

\newtheorem{theorem}{Theorem}[section]
\newtheorem{prop}{Proposition}[section]
\newtheorem{lemma}{Lemma}[section]
\newtheorem{cor}{Corollary}[section]

\newtheorem{definition}{Definition}[section]

\title[Range characterization]{On the range of the attenuated Radon transform in strictly convex sets}
\begin{document}
\date{}
\author{Kamran Sadiq}
\address{Department of Mathematics, University of Central Florida, Orlando, Florida 32816}
\email{ksadiq@knights.ucf.edu, tamasan@math.ucf.edu}

\author{Alexandru Tamasan}

\date{}


\begin{abstract}
We present new necessary and sufficient conditions for a function on $\partial\Omega\times S^1$
to be in the range of the attenuated Radon transform of a sufficiently smooth function support in the convex set $\overline\Omega\subset\mathbb{R}^2$. The approach is based on an explicit Hilbert transform associated with traces of the boundary of A-analytic functions in the sense of Bukhgeim.
\end{abstract}

\maketitle

\section{Introduction} \label{S:intro}
In this paper we are concerned with the range characterization of the
attenuated Radon transform of function of compact support in the plane. Necessary and sufficient constraints on range of the non-attenuated (classical) Radon transform in the Euclidean space have been known since the works in \cite{gelfandGraev}, \cite{helgason}, and \cite{ludwig}.
These constraints, known as the Cavalieri or the moment conditions, are in terms of the angular variable. For function in the Schwartz class, they are essentially unique due to a Paley-Wiener type theorem. Moreover, the Helgason support theorem  extends the conditions to smooth functions of compact support \cite{helgasonBook}. However, in the case of functions of compact support, it is possible to obtain essentially different range conditions since more than one operator can annihilate functions of compact support in the range of the Radon transform. The results here constitute one such example.

Inversion methods of the attenuated Radon transform in the plane appeared first in \cite{ABK}, and \cite{novikov01}, and various developments can be found in \cite{naterrer01}, \cite{bomanStromberg}, \cite{finch}, \cite{bal}. The interest in range conditions stems out from their applications to data enhancement in medical imaging methods such as Single Photon, or Positron Emission Computed Tomography \cite{naterrerBook}.
For the Euclidean attenuated Radon transform, some range characterization based on the inversion procedure in \cite{novikov01} can be found in \cite{novikov02}. These constraints are also in terms of the angular variable.

Different from the existing results above, our new characterization is in terms of a Hilbert transform associated with the A-analytic maps \`{a} la Bukhgeim \cite{bukhgeim_book}, and represents constraints in the spatial variable; see Theorems \ref{Th3}, and \ref{AttRTTh}. Range characterizations in terms of a Hilbert type transform were first introduced in \cite{pestovUhlmann} in the non-attenuated case for smooth functions on two dimensional compact simple Riemmanian manifolds. Extensions to the attenuated case and to tensor tomography has been recently obtained in \cite{paternainUhlmann}.

Throughout this work $\Omega\subset\mathbb{R}^2$ is a convex bounded domain with $C^2$-smooth boundary $\Gamma$ with strictly positive curvature bound.
Let $a, f \in C(\ol \OM)$ be extended by zero outside. The {\em divergence beam transform} of $a$ is defined as
\begin{align}\label{divbeam}
Da(x,\tta): =\int_{0}^{\INF} a(x+t\tta)dt,
\end{align}and the {\em attenuated Radon transform} of $f$  (with attenuation $a$) by
\begin{align}\label{AttRT}
\int_{-\infty}^\infty f(x+t\theta)e^{-Da(x+t\theta,\theta)}dt.
\end{align}
The integral in \eqref{AttRT} is constant in $x$ in the direction of $\tta$, and
this defines a function on the cotangent bundle of the circle $\sph$. In here
however we describe the range characterization in terms of function on $\Gam \times \sph$ and it make sense to think of integral in \eqref{AttRT} defined on $\ol\OM \times \sph$.

For any $(x,\theta)\in\ol\Omega\times S^1$, let $\tau_\pm(x,\theta)$ denote the distance from $x$ in the $\pm\theta$ direction to the boundary, and distinguish the endpoints $x^\pm_\theta\in\Gamma$ of the chord in the direction of $\theta$ passing through $x$ by
\begin{equation}\label{xthetapm}
x^\pm_\theta:=x\pm\tau_\pm(x,\theta)\theta.
\end{equation}Note that
\begin{align}\label{chord}
\tau(x,\theta)=\tau_+(x,\theta)+\tau_-(x,\theta)
\end{align}is the length of the chord.

\begin{figure}[ht]
\centering
\begin{tikzpicture}[scale=1.5,cap=round,>=latex]

 \draw[thick] (0cm,0cm) circle(1cm);
 \draw[gray] (0cm,0cm) -- (45:1cm);
 \filldraw[black] (45:1cm) circle(1.2pt);
  \coordinate[label=right:$x_{\theta}^{+}$] (x_{+}) at (45:1.01cm);

 \draw[->] (45:1.2cm) -- (45:1.6cm);
 \coordinate[label=left:$\theta$] (theta1) at (45:1.4cm);

 \filldraw[black] (45:0.2cm) circle(1.2pt);
  \coordinate[label=right:$x$] (x) at(45:0.2cm);
  \coordinate[label=above:$\Omega$] (OM) at(120:0.6cm);

  \draw[gray] (0cm,0cm) -- (225:1cm);
  \filldraw[black] (225:1cm) circle(1.2pt);

  \coordinate[label=below:$x_{\theta}^{-}$] (x_{-}) at (225:1cm);

  \draw[->] (225:1.6cm) -- (225:1.2cm);
  \coordinate[label=left:$\theta$] (theta2) at (225:1.2cm);

  \tikzset{
    position label/.style={
       above = 3pt,
       text height = 1.5ex,
       text depth = 1ex
    },
   brace/.style={
     decoration={brace,mirror},
     decorate
   }
}


\draw [brace,decoration={raise=0.5ex}] [brace]  (x_{+}.north) -- (x.north) node [position label, pos=0.6, rotate = 45,scale=0.8] {$\tau_{+}(x,\tta)$};


 \tikzset{
    position label/.style={
       below = 3pt,
       text height = 1.5ex,
       text depth = 1ex
    },
   brace/.style={
     decoration={brace,mirror},
     decorate
   }
}

 \draw [brace,decoration={raise=0.5ex}] (x_{-}.north) -- (x.north) node [position label, pos=0.55, rotate = 45, scale=0.8] {$\tau_{-}(x,\tta)$};

\end{tikzpicture}
\caption{} \label{fig:1}
\end{figure}

\begin{definition}\label{AttRadon_definition}
We say that $g$ on $\Gamma\times\sph$ {\em is an attenuated Radon transform of} $f$ with attenuation $a$, if
\begin{align}\label{Radon_definition}
g(x^+_\theta,\theta)-\left[e^{-Da} g\right](x^-_\theta,\theta)= \int_{\tau_-(x,\theta)}^{\tau_+(x,\theta)} f(x+t\theta)e^{-Da(x+t\theta,\theta)}dt,
\end{align}for every $(x,\theta)\in\ol\OM\times \sph$.
\end{definition}

It is easy to see that $g$ in definition \ref{AttRadon_definition} is not unique, since we can add to $g$ any function $h$ on $\Gamma\times S^1$ such that
\begin{equation}\label{nonuniqueness}
h(x^+_\theta,\theta)=\left[e^{-Da}h\right](x^-_\theta,\theta).
\end{equation}

{\em If $g$ is an attenuated Radon transform in the sense above, we use the notation $g\in R_af$. In the case $a\equiv 0$, we use the notation $g\in Rf$.}

We note that \eqref{nonuniqueness} is the only way non-uniqueness occurs, and that, for functions defined in the whole plane with Radon data at infinity, such an ambiguity cannot occur.

In motivation to our definition \eqref{Radon_definition} we note that the function $g$ is precisely the trace on $\Gamma\times S^1$ of
solutions $u$  to the transport equation
\begin{equation} \label{eq:TEq2}
    \tta\cdot\nabla u(x,\tta) +a(x) u(x,\theta) = f(x), \quad  (x,\tta)\in \OM \times S^{1},
\end{equation}
in accordance to the physical model of transport, where $u(x,\theta)$ is the density of particles at $x$ moving in the direction $\theta$, $f(x)$ is the density of radiating particles per unit path-length, and $a(x)$ is the medium capability of absorption per unit path-length at $x$.

Our main result gives necessary and sufficient conditions for $g\in R_a f$. These conditions characterize the traces $u|_{\Gamma\times S^1}$ of solutions of \eqref{eq:TEq2}, as traces on $\Gamma$ of solutions of A-analytic functions. For the sake of clarity we first treat the non-attenuated case $a\equiv 0$, and then reduce the attenuated case to it.

\section{Preliminaries}
In this section we recall some preliminary notions and results from the theory of A-analytic sequence valued maps, singular integral and harmonic analysis, and set notations. We justify the results which are new.

We treat first the non-attenuated case ($a\equiv 0$), in which the transport equation further simplifies to
\begin{equation} \label{eq:TEql}
    \tta\cdot\nabla v(x,\tta) = f(x), \quad  (x,\tta)\in \OM \times S^{1}.
\end{equation}

With the complex notations
\[z = x_{1} + ix_{2}, \quad \ol{\del} = \left( \del_{x_{1}}+i \del_{x_{2}} \right) /2 ,\quad \del = \left( \del_{x_{1}}- i \del_{x_{2}} \right) /2,\]
the advection operator becomes $$\tta\cdot\nabla = e^{-i\fii} \ol{\del} + e^{i\fii} \del,$$
where $\fii=\arg(\tta)$ denotes an angular variable .

Let
$v(z,\theta) = \sum_{-\infty}^{\infty} v_{n}(z) e^{in\fii},$ be the (formal) Fourier expansion of $v$ in the angular variable. Provided appropriate convergence of the series as specified in the theorems, we see that $v$ solves (\ref{eq:TEql}) if and only if its Fourier coefficients solve
\begin{equation} \label{eq:TEq3}
\ol{\del} v_{-1}(z) + \del v_{1}(z) = f(z),
\end{equation}
and, for $n\neq 1$,
$$ \ol{\del} v_{n}(z) + \del v_{n-2}(z) = 0.$$

Since $v$ is real-valued, its Fourier coefficients appear in complex-conjugate pairs, $\ol{v_{n}} = {v_{-n}},$  so
that it suffices to work with the sequence of non-positive indexes (this choice preserves the original notation in \cite{bukhgeim_book}).

\begin{definition}The sequence valued map $z\mapsto  \bv(z): = \langle v_{0}(z),v_{-1}(z),v_{-2}(z),... \rangle$
is called A-analytic if $\bv\in C(\ol\OM;l_\INF)\cap C^1(\OM;l_\INF)$ and
\begin{equation}\label{Aanalytic}
\ol{\del} v_{n}(z) + \del v_{n-2}(z) = 0, \quad n=0,-1,-2,...
\end{equation}
\end{definition}

For a compact set $K\subset\mathbb{R}^2$, such as $\Gamma, \ol\OM$, $\sph$, or $\ol\OM\times\sph$, by $C^\alpha(K)$ we denote the Banach space of uniform $\alpha$- H\"{o}lder continuous functions endowed with the norm
$$\|f\|_{C^\alpha(K)}:=\sup_{z\in K}|f(z)|+\sup_{z,w\in K,\,z\neq w}\frac{|f(z)-f(w)|}{|z-w|^\alpha}.$$
By $C^\alpha(\OM)$ we denote the space of locally uniform $\alpha$- H\"{o}lder continuous functions.


We note the general fact that, for a sequence of nonnegative numbers.
\begin{lemma}\label{seqresult}
Let $\{c_n\}$ be a sequence of nonnegative numbers. Then
\begin{align*}
(i)\quad  \sum_{k=1}^{\INF} \sum_{n=0}^{\INF} k \; c_{k+n} &= \sum_{j=1}^{\INF} \frac{j(j+1)}{2} \; c_{j}, \\ \nn
(ii)\quad \sum_{k=1}^{\INF} \sum_{n=0}^{\INF} c_{k+n} &= \sum_{j=1}^{\INF} j\; c_{j},
\end{align*}
whenever one of the sides in (i) and (ii) is finite.
\begin{proof}
(i) Indeed, if we introduce the change of index $j=k+n$, for $k\geq 1$, ($j-n \geq 1,$ and $n \leq j-1$) we get
\[\sum_{k=1}^{\INF} \sum_{n=0}^{\INF} k c_{k+n} = \sum_{j=1}^{\INF} \sum_{n=0}^{j-1} (j-n) \; c_{j} =   \sum_{j=1}^{\INF} c_{j} \sum_{n=0}^{j-1} (j-n) = \sum_{j=1}^{\INF} \frac{j(j+1)}{2} c_{j}.\]

(ii) Indeed the change of index $j=k+n$ for $k\geq 1$ yields
\[\sum_{k=1}^{\INF} \sum_{n=0}^{\INF}c_{k+n} = \sum_{j=1}^{\INF} \sum_{n=0}^{j-1}c_{j} =   \sum_{j=1}^{\INF} c_{j} \sum_{n=0}^{j-1} 1
= \sum_{j=1}^{\INF} jc_j.\]
\end{proof}
\end{lemma}

In several of the arguments we make use of the following Bernstein's lemma below (see, e.g., \cite{katznelson}).

\begin{lemma}\label{bernstein_lemma} Let $f\in C^{k,\alpha}(\sph)$, $\alpha>1/2$, and $\{\hat{f}_n\}$ be the sequence of its Fourier coefficients. Then
$\displaystyle\sum_{n=-\infty}^\infty\; |n|^{k}|\hat{f}_n|<\infty.$
\end{lemma}

To characterize traces of A-analytic functions we need to control the decay in the Fourier terms. We work in the following Banach spaces
\begin{equation} \label{lGamdefn}
 l^{1,1}_{\INF}(\Gamma): = \left \{ \bv=  \langle v_0, v_{-1}, ... \rangle  : \sup_{w\in \Gam}\sum_{j=1}^{\INF}  j \lvert v_{-j}(w) \rvert < \INF \right \},
\end{equation}
and
\begin{equation} \label{CepsGamdefn}
 C^{\eps}(\Gamma ; l_1) := \left \{ \bv =  \langle v_0, v_{-1}, ... \rangle :
\sup_{\xi\in \Gam} \lVert \bv(\xi)\rVert_{\ds l_{1}} + \underset{{\substack{
            \xi,\eta \in \Gam \\
            \xi\neq \eta } }}{\sup}
 \frac{\lVert \bv(\xi) - \bv(\eta)\rVert_{\ds l_{1}}}{|\xi - \eta|^{ \eps}} < \INF \right \},
\end{equation}where $l^1$ is the space of sumable sequences. By replacing $\Gamma$ with $\overline\OM$ and $l_{1}$ with $l_{\INF}$ in \eqref{CepsGamdefn} we similarly define  $C^{\eps}(\overline\OM ; l_1)$, respectively, $C^{\eps}(\ol\OM ; l_{\INF})$, where $l_{\INF}$ denotes the space of bounded sequences.

We describe next the two operators which define the Hilbert transform associated with A-analytic maps.
For $\bv\in C^\eps(\Gamma,l_1)$, we consider the Cauchy integral operators defined componentwise by
\begin{equation} \label{CauchyIntEq}
(C\bv)_{n}(\xi) := (Cv_n)(\xi) =\frac{1}{2\pi i} \int_{\Gam } \frac{v_{n}(w)}{w - \xi} dw,\quad\xi\in\OM,
\end{equation}and
\begin{equation} \label{CauchyIntBd}
(S\bv)_{n}(\xi) := (Sv_n)(\xi)=\frac{1}{\pi i} \int_{\Gam } \frac{v_{n}(w)}{w - \xi} dw, \quad\xi\in\Gamma,\quad n=0,-1,-2,...
\end{equation}
The later integral is understood in the Cauchy principal value sense.

The following result is a componentwise extension of Sokhotski-Plemelj formula (e.g., \cite{muskhellishvili}) to sequence valued maps.
\begin{prop}[Sokhotski-Plemelj]\label{Plemeljpropext}
Let $\bv \in C^\eps(\Gam;l_1)$ as in \eqref{CepsGamdefn}.
Then, for every $\xi_0\in\Gamma$, the limit
\begin{equation}\label{plemelj_l1}
\lim_{\OM\ni \xi \to \xi_{0}}\left\| (C\bv) (\xi) - \frac{1}{2} \bv (\xi_0) - \frac{1}{2} S\bv(\xi_0)\right\|_{l_1}=0
\end{equation}
defines an extension of $C\bv$ from $\OM$ to $\ol{\OM}$ as a
$Holder$ continuous map with values in $l_{1}$, i.e,
\begin{equation}\nn
C : C^\eps(\Gam ; l_{1})\longrightarrow C^{ \eps}\left( \ol{\OM}; l_{1} \right) \cap C^{1}\left( \OM; l_{1} \right).
\end{equation}
\end{prop}

The fact that $C\bv \in C^1 (\OM ; l_{1})$ follows directly from the
local character of differentiability and from the fact that
$\sum_{n=1}^{\INF}\int_{\Gam}  |v_{-n}(w)dw|  < \INF $.

Next we introduce the second operator which appears in the definition of the Hilbert transform. It is defined componentwise for each index $n \leq 0$,  $\xi \in \overline\OM$, $w\in \Gam$, and $\bv \in l^{1,1}_{\INF}(\Gamma)$ by
\begin{equation} \label{Gdefn}
(G \bv)_{n}(\xi) = \frac{1}{\pi i} \int_{\Gam } \left \{ \frac{dw}{w-\xi}-\frac{d \ol{w}}{\ol{w-\xi}} \right \} \sum_{j=1}^{\infty} \ds v_{n-2j}(w) \left ( \frac{\ol{w-\xi}}{w-\xi} \right) ^{j}.
\end{equation}
We will use the following mapping property of $G$.
\begin{prop}\label{Gprop}
\begin{equation}\nn
G :  C^\eps(\Gam ; l_{1})\cap l^{1,1}_{\INF}(\Gamma) \longrightarrow C^{\eps}\left( \ol{\OM}; l_{\INF} \right) \cap C^{1}\left( \OM; l_{\INF} \right).
    \end{equation}
\end{prop}
\begin{proof}
Let $ \xi, \, \xi_{0} \in \OM$. Since $\bv \in l^{1,1}_{\INF}(\Gamma)$, it follows from \eqref{Gdefn} that each component $(G \bv)_{n}(\xi)$  is well-defined for $n \leq0$.

Now let $ w(\fii) = \xi \; + \; l_{\xi}(\fii)  e^{i \fii}$ be a parametrization of $\Gamma$, where $l_{\xi}(\fii) = |\xi - w(\fii)| $. Since the boundary $\Gamma$ is at least $C^1$, we have that $\xi\mapsto l_\xi$ is Lipschitz in $\ol\OM$ uniformly in $\fii\in [0,2\pi]$, i.e.,
\begin{equation}\label{lipshitz_constant}
\left| l_{\xi}(\fii) - l_{\xi_{0}}(\fii) \right|\leq L |\xi - \xi_{0}|,
\end{equation}for some constant $L>0$. Moreover,
\begin{equation*}
\frac{dw}{w-\xi} = \left[ \frac{l_{\xi}'}{l_{\xi}} + i \right] d\fii ,\quad
\frac{d\ol{w}}{\ol{w-\xi}} = \left[ \frac{l_{\xi}'}{l_{\xi}} - i \right] d\fii,  \quad \left ( \frac{\ol{w-\xi}}{w- \xi} \right) = e^{-2i\fii},
\end{equation*}
and note that the measure
$\displaystyle\frac{dw}{w-\xi} -\frac{d\ol{w}}{\ol{w-\xi}}=2id\fii$ in \eqref{Gdefn} is nonsingular.

For each integer $n\leq0$, the equation \eqref{Gdefn} rewrites
\begin{equation}\nn
 (G \, \bv)_{n}(\xi) = \ds \frac{2}{\pi } \int_{0}^{2\pi}
\sum_{j=1}^{\INF}  g_{n-2j}(\xi + l_{\xi}(\fii) \, e^{i \fii}) \; e^{-2ij\fii} d\fii,\; \xi\in\overline\OM.
\end{equation}

Since $\bv\in C^{\eps}(\Gam ; l_{1})$, we have
\begin{equation}\nn
\kappa:=  \underset{{\substack{
            w_{1},w_{2} \in \Gam \\
            w_{1} \neq w_{2} } }}{\sup} \sum_{n=0}^{\INF} \, \frac{ \ds \lvert g_{-n}(w_{1}) - g_{-n}(w_{2}) \rvert}{ \ds \lvert w_1 - w_2 \rvert^{\displaystyle \eps}}< \INF.
\end{equation}
We estimate for each $n$,
\begin{align*}
\lvert(G\bv)_{n}(\xi)& - (G \bv)_{n}(\xi_{0})\rvert\\
& \leq\frac{2}{\pi}  \sum_{j=1}^{\INF}
\int_{0}^{2\pi} \left |v_{n-2j}(\xi + l_{\xi}(\fii) \, e^{i \fii}) - v_{n-2j}(\xi_{0} + l_{\xi_{0}}(\fii) \, e^{i \fii}) \,\right | d\fii \\ \nn
& \leq
\frac{2\kappa}{\pi}
\int_{0}^{2\pi} \left | (\xi - \xi_{0})  + \left [ l_{\xi}(\fii) - l_{\xi_{0}}(\fii) \right ]  \, e^{i \fii} \right |^{\eps} d\fii, \\ \nn
& \leq
\frac{2\kappa}{\pi}
\int_{0}^{2\pi} \; \left ( 2 \left| \xi - \xi_{0} \right |^{\eps} + \left| l_{\xi}(\fii) - l_{\xi_{0}}(\fii) \right|^{\eps}
\right )  d\fii,  \\ \nn
& \leq
( 8\kappa + 4\kappa L^\eps )\left| \xi - \xi_{0} \right|^{\eps}.
\end{align*}
In the third inequality above we used $|a+b|^{\eps} \leq 2|a|^{\eps} + |b|^{\eps}$, and the fourth inequality uses \eqref{lipshitz_constant}.

Next we show that $G\bv\in C^{1}(\OM ; l_{\INF})$. Suffices to carry
the estimates in the neighborhood $\ol{B(\xi_0,r_0)}\subset\OM$ of an arbitrary point $\xi_0 \in \OM $, where $r_0=\dist(\xi_0,\Gamma)/2>0$.

For $\xi \in B(\xi_0,r_0)$ arbitrary we have
\begin{align}\label{gradientEstimate}
\left | \nabla_{\xi} \left \{ \frac{dw}{w-\xi}-\frac{d \ol{w}}{\ol{w-\xi}} \right \} \right | &= \left |  2 \im \left ( \frac{dw}{{(w-\xi)}^{2}} \right )  \right |\leq c|dw|,
\end{align}where $c= 2/{r_0^2}$.

For each $n\leq 0$, we have
\begin{align*}
\nabla_{\xi} (G \, \bv)_{n}(\xi)
&= \frac{1}{\pi i} \int_{\Gam } \nabla_{\xi} \left \{ \frac{dw}{w-\xi}-\frac{d \ol{w}}{\ol{w-\xi}} \right \} \sum_{j=1}^{\INF} \ds v_{n-2j}(w) \left ( \frac{\overline{w-\xi}}{w-\xi} \right) ^{j}  \\ &\qquad \quad {}  + \frac{1}{\pi i}\int_{\Gam }  \left \{ \frac{dw}{w-\xi}-\frac{d \ol{w}}{\ol{w-\xi}} \right \} \sum_{j=1}^{\INF} \ds v_{n-2j}(w)  \nabla_{\xi} \left ( \frac{\ol{w-\xi}}{w-\xi} \right) ^{j}.
\end{align*}

For  $\bv \in C^\eps(\Gam ; l_{1})\cap l^{1,1}_{\INF}(\Gamma)$, and $\xi\in \ol{B(\xi_0,r_0)}$, the right hand side above is bounded uniformly in $n$, since
\begin{align*}
 \left | \nabla_{\xi} (G \, \bv)_{n}(\xi) \right |
 &\leq \frac{ c}{\pi} \int_{\Gam } \sum_{j=1}^{\infty} \ds \left | v_{n-2j}(w) \right | dw   \\
 &+  \frac{ c}{\pi} \int_{\Gam } \sum_{j=1}^{\INF} \ds j \, \left | v_{n-2j}(w) \right | dw <\infty.
 \end{align*}
Therefore $G\bv\in C^1(\OM;l_\infty).$
\end{proof}
\section{The Hilbert transform of A-analytic maps} \label{S:SeqVResults}
In this section we introduce the Hilbert transform  $\HT_0$  associated with the traces on $\Gamma$ of A-analytic maps in $\OM$.

Recall the operator $S$ and $G$ as defined in \eqref{CauchyIntBd}, and \eqref{Gdefn}.
\begin{definition}\label{hilbertT_definition}
The Hilbert transform  $\HT_0$ for $\bg=\langle g_0,g_{-1},...\rangle \in l^{1,1}_{\INF}(\Gamma)\cap C^\eps(\Gamma;l_1)$ is defined by
\begin{align}\label{hilbertT}
\HT_0\bg:=i[S+G]\bg,
\end{align} and written componentwise, for $n=0,-1,-2,...$, as
\begin{align*}
(\HT_0\bg)_n(\xi)=&\frac{1}{\pi} \int_{\Gam } \frac{g_{n}(w)}{w - \xi} dw\\
&+\frac{1}{\pi} \int_{\Gam } \left \{ \frac{dw}{w-\xi}-\frac{d \ol{w}}{\ol{w-\xi}} \right \} \sum_{j=1}^{\infty} \ds g_{n-2j}(w)
\left( \frac{\ol{w-\xi}}{w-\xi} \right) ^{j},\quad \xi\in\Gamma.
\end{align*}
\end{definition}
The mapping properties of $S$, and $G$ in Propositions \ref{Plemeljpropext} and \ref{Gprop}, together with the continuous embedding of $l_1\subset l_\infty$, yields
\begin{prop}\label{Hproperties}
\begin{equation}
\HT_0 :  C^\eps(\Gam ; l_{1})\cap l^{1,1}_{\INF}(\Gamma) \longrightarrow C^{\eps}\left( \Gamma; l_{\INF} \right),
\end{equation}is a continuous map.
\end{prop}

The name of this transform will be motivated in the next section, where we show that traces on $\Gamma$ of A-analytic maps lie in the kernel of $[I+i\HT_0]$ in analogy with the classical Hilbert transform for analytic functions.

At the heart of the theory of A-analytic maps lies a Cauchy integral formula. A class of such Cauchy integral
formulae were first introduced by Bukhgeim in \cite{bukhgeim_book}. The explicit form \eqref{vnDefn} below is due to Finch \cite{finch}; see also \cite{tamasan02, tamasan03, tamasan07} where one works with square summable sequences.

\begin{theorem}\label{Th1}
Let $\bg=\langle g_0,g_{-1},...\rangle \in l^{1,1}_{\INF}(\Gamma)\cap C^\eps(\Gamma;l_1)$ be a sequence valued map defined at the boundary $\Gamma$.
For $\xi\in\OM$ and each index  $n \leq 0$ we define $v_n(\xi)$ by
\begin{equation} \label{vnDefn}
v_{n}(\xi) := \frac{1}{2}(G \bg)_{n}(\xi) + (C \bg)_{n}(\xi).
\end{equation}
Then $\bv:=\langle v_0,v_{-1},...\rangle \in C^{1,\eps}(\OM;l_\infty)$, and, for each $n=0,-1,...,$
\[\ol{\partial}v_{n}(\xi) +\del v_{n-2}(\xi) =0,\quad\xi\in\Omega.\]
Moreover, for each $n=0,-1,-2,...$, the component $v_n$ extends continuously to $\ol{\OM}$ with limiting values
\begin{equation}
v_{n}^{+}(\xi_{0}):=  \underset{\OM \owns \xi \to \xi_{0} \in \Gam }{\lim} v_{n}(\xi),
\end{equation}
where
\begin{equation} \label{vn+Defn}
v_{n}^{+}(\xi_{0}) = \frac{1}{2} (G\bg)_{n}(\xi_0) + \frac{1}{2}(S+I)g_{n}(\xi_0).
\end{equation}
\end{theorem}

\begin{proof}
Let $\xi \in \OM$ and $n\leq 0$ arbitrarily fixed. Since $\bg \in l^{1,1}_{\INF}(\Gamma)\cap C^\eps(\Gamma,l_1)$, both $(G \bg)_{n}(\xi)$ and  $(C\bg)_{n})(\xi)$ are well-defined. Moreover, from Propositions \ref{Plemeljpropext} and \ref{Gprop}
we have that $ \bv\in C(\ol\OM;l_\INF)\cap C^1(\OM;l_\INF)$.

For each $n\leq 0$, by its definition in  \eqref{vnDefn}, we have
\begin{align*}
2 \pi i v_{n}(\xi) &= \sum_{j=0}^{\infty} \int_{\Gam}
\frac{ g_{n-2j}(w)\ol{(w-\xi)}^{j}}{(w-\xi)^{j+1}}dw
 \\
 & \qquad - \sum_{j=1}^{\INF} \int_{\Gam}
\frac{ g_{n-2j}(w)\ol{(w-\xi)}^{j-1}}{(w-\xi)^{j}}d\ol{w}.
\end{align*}

From where
\begin{align}\label{del_v_n}
2\pi i\del v_{n-2}(\xi)=&\sum_{j=1}^{\INF} \int_{\Gam}
\frac{ jg_{n-2j}(w) \ol{(w-\xi)}^{j-1}}{(w-\xi)^{j+1}}dw\nn\\& -
\sum_{j=2}^{\INF} \int_{\Gam}\frac{(j-1) g_{n-2j}(w)\ol{(w-\xi)}^{j-2}}{(w-\xi)^{j}}d\ol{w},
\end{align}and
\begin{align}\label{dbar_v_n}
2\pi i\ol{\del} v_{n}(\xi) = &- \sum_{j=1}^{\INF} \int_{\Gam}
\frac{ j g_{n-2j}(w) \ol{(w-\xi)}^{j-1}}{(w-\xi)^{j+1}}dw \nn\\&+ \sum_{j=2}^{\INF} \int_{\Gam}
\frac{(j-1) g_{n-2j}(w)\ol{(w-\xi)}^{j-2}}{(w-\xi)^{j}}d\ol{w}.
\end{align}
By summing \eqref{del_v_n} and \eqref{dbar_v_n} we obtain $ \ol{\del} v_{n} + \del v_{n-2} = 0$ for each $ n=0,-1,-2,...$.

The regularity $v_n\in C^{1,\eps}(\OM)$ follows from the explicit formula \eqref{del_v_n}, and the fact that
$\displaystyle \xi\mapsto\frac{\ol{(w-\xi)}^{j}}{(w-\xi)^{j+2}}$ are locally uniform $\eps$-H\"{o}lder continuous.

The continuity to the boundary are consequences of Propositions \ref{Plemeljpropext} and \ref{Gprop}.
In the limits below $\xi\in\OM$, and
$\xi_0 \in \Gam$:
\begin{align*}
\underset{ \xi \to \xi_{0} }{\lim} v_{n}(\xi) &= \underset{ \xi \to \xi_{0} }{\lim}
\frac{1}{2} (G \bv)_{n}(\xi) \,+\,  \underset{ \xi \to \xi_{0} }{\lim} (C\bg)_{n})(\xi) \\
&=\frac{1}{2}(G\bv)_{n}(\xi_0) \,+\,  \frac{1}{2}  g_{n}(\xi_0) + \frac{1}{2} \sum_{n=0}^{\INF}(S\bg)_{n}(\xi_0)\\
&=\frac{1}{2}(G\bv)_{n}(\xi_0) + \frac{1}{2}(S+I)g_{n}(\xi_0).
\end{align*}
\end{proof}

The following theorem presents necessary and sufficient conditions for sufficiently regular sequence valued map to be the trace at the boundary of an $A$-analytic function.

\begin{theorem}\label{NecSuf}
Let $\bg =\langle g_0, g_{-1}, g_{-2}, \cdots \cdot \rangle \in l^{1,1}_{\INF}(\Gamma)\cap C^\eps(\Gamma,l^1)$.
For $\bg$ to be boundary value of an $A$-analytic function it is necessary and sufficient that
\begin{equation} \label{NecSufEq}
 (I+i\HT_0) \bg=0.
 \end{equation}
\end{theorem}
\begin{proof}

For the necessity let $\bv$ be $A$-analytic as in \eqref{Aanalytic} whose trace $\bv|_{\Gamma}=\bg$, in the sense that
\begin{equation*}
\underset{\OM \owns \xi \to \xi_{0} \in \Gam }{\lim} v_{n}(\xi) = {g_{n}}(\xi_{0}), \quad  n \leq 0.
\end{equation*}
By $\eqref{vn+Defn}$ we obtain
\begin{equation} \nonumber
g_{n}(\xi_{0}) = \frac{1}{2} (G\bv)_{n}(\xi_0) + \frac{1}{2}Sg_{n}(\xi_0) + \frac{1}{2} g_{n}(\xi_0),
\end{equation}or,
\begin{align}\label{equivCharac}
[(I - S-G)\bg]_{n} = 0, \quad n \leq 0.
\end{align}Since $\HT_0=i[S+G]$, \eqref{equivCharac} is a componentwise representation of \eqref{NecSufEq}.

Next we prove sufficiency. Let $\bg \in l^{1,1}_{\INF}(\Gamma)\cap C^\eps(\Gamma,l_1)$ satisfy\eqref{NecSufEq}, and define $\bv$ in $\OM$  by the Cauchy Integral formula \eqref{vnDefn}.
From Propositions  \ref{Plemeljpropext} and \ref{Gprop} we have that $\bv\in C^{1}(\OM;l_\INF) \cap C(\ol{\OM};l_\INF)$, and from
Theorem \ref{Th1} we see that $\ol{\del} v_{n} + \del v_{n-2} = 0$, for each $n \leq 0.$
Therefore $\bv$ is $A$-analytic. Moreover,
\begin{align*}
\underset{\OM \owns \xi \to \xi_{0} \in \Gam }{\lim} v_{n}(\xi)
&= \frac{1}{2} (G\bv)_{n}(\xi_0) + \frac{1}{2}(S+I)g_{n}(\xi_0)\\
&= \frac{1}{2} (I - S)g_{n}(\xi_0) + \frac{1}{2}(S+I)g_{n}(\xi_0)\\
&= g_{n}(\xi_0),
\end{align*}where the first equality uses \eqref{vn+Defn}, whereas the second equality uses \eqref{NecSufEq}.
\end{proof}

\section{Range characterization of the Radon transform} \label{S:MainResults}
This section concerns our main result in the non-attenuated case ($a\equiv 0$). The results require a stronger topology. For $\eps>0$,  we consider the space
$Y_{\eps} = C^{\eps}(\Gam ; l^{1,1}(\sph)) \cap C^{0}(\Gam ; l^{1,2}(\sph))$ i.e
\begin{equation} \label{YGamdefn}
 Y_{\eps} = \left \{ \bg\in l^{1,2}_{\INF}(\Gamma) :  \underset{{\substack{
            \xi, \mu  \in \Gam \\
            \xi \neq \mu } }}{\sup} \sum_{j=1}^{\INF} j\frac{\lvert g_{-j}(\xi) - g_{-j}(\mu) \rvert}{|\xi - \mu|^\eps} < \INF \right \},
\end{equation}where
\begin{equation} \label{lGamdefnY}
 l^{1,2}_{\INF}(\Gamma):= \left \{ \bg =  \langle g_0, g_{-1}, g_{-2}, \cdots \cdot \rangle \; : \sup_{w\in \Gam} \sum_{j=1}^{\INF} j^{2} \,\lvert g_{-j}(w) \rvert < \INF \right \}.
\end{equation}

For the sake of clarity in the statement of the main result we introduce the following projections.
\begin{definition}
Given a function $g \in C(\ol \OM; L^{1}(\sph))$, we consider the projections \begin{align}\label{Pminusplus}
\Pminus (g) := \langle g_0, g_{-1}, g_{-2}, ... \rangle, \qquad \Pplus (g) := \langle g_0, g_{1}, g_{2}, ... \rangle,
\end{align}
where $g_{n}(z) = \frac{1}{2 \pi} \int_{0}^{2 \pi} g(z,\tta) e^{-i n \fii} d\fii,$ for $z\in \ol \OM$, is the $n-th$ Fourier coefficients for $n \in \mathbb{Z}$.
Conversely, given $\bg(z) = \langle g_0(z), g_{-1}(z), g_{-2}(z), ... \rangle \in C(\ol \OM; l^{1})$, we define a corresponding real valued function $g$ on $\ol \OM \times \sph$ by
\begin{align}\label{PstarOp}
\Pstar (\bg) := g_0(z)+ 2 \re\left(\sum_{n = 1}^{\INF} g_{-n}(z)e^{-i n\fii}\right).
\end{align}
\end{definition}

The properties below are immediate:

If $g$ is a function on $\Gam \times \sph$ then
\begin{align}
&(i) \; \Pminus \Pstar \Pminus (g) = \Pminus (g) , \\
&(ii)\; \Pminus (e^{\pm h} g) = (\Pplus e^{\pm h}) \ast_{n} (\Pminus (g)), \label{PmpsProp}
\end{align}
where $\ast_{n}$ is the convolution operator on sequences and $h$ is a function on $\Gam \times \sph$ with only non negative Fourier modes.

The following result gives some of the properties of the $P^\pm$ and $\Pstar$ operators. Recall the definition of the space $Y_{\eps}$ in \eqref{YGamdefn}.

\begin{prop}\label{gCepsgbarXY}Let $\alpha>1/2$, and $\eps>0$ be arbitrarily small. Then

(i) $\Pminus: C^{\eps} \left ( \Gam; C^{1,\alpha}(\sph) \right)\to l^{1,1}_{\INF}(\Gamma)\cap C^\eps(\Gamma;l^1)$,

(ii) $\Pminus: C^{\eps} \left ( \Gam ; C^{1,\alpha}(\sph)\right) \cap C^0\left ( \Gam;C^{2,\alpha}(\sph)\right)\to Y_{\eps}$,

(iii) $\Pstar: C^{1,\alpha}(\OM;l^{1}) \cap C^{\alpha}(\ol \OM;l^{1})\to C^{1,\alpha} (\OM \times \sph) \cap C^{\alpha}( \ol \OM \times \sph )$.
\end{prop}

\begin{proof}
Let $g\in C^{\eps} \left ( \Gam; C^{1,\alpha}(\sph) \right)$. Then
\begin{equation} \label{gCAlphaNorm}
\ds \underset{\xi \in \Gam}{\sup} \, \lVert g(\xi,\cdotp) \rVert_{C^{1,\alpha}}
+ \underset{{\substack{
            \xi, \mu  \in \Gam \\
            \xi \neq \mu } }}{\sup} \, \frac{ \ds \lVert g(\xi,\cdotp) - g(\mu,\cdotp) \rVert_{C^{1,\alpha}}}{ \ds \lvert \xi - \mu \rvert^{\eps}} < \INF.
\end{equation}

From
\begin{equation}\label{gBernProp}
\underset{\xi \in \Gam}{\sup} \, \sum_{j=1}^{\INF}\, j \, \lvert g_{-j}(\xi) \rvert\leq
 \underset{\xi \in \Gam}{\sup} \, \lVert g(\xi,\cdotp) \rVert_{C^{1,\alpha}}<\infty,
\end{equation}
and by Lemma \ref{bernstein_lemma}, $\Pminus (g)\in l^{1,1}_{\INF}(\Gamma)$.

Another application of Lemma \ref{bernstein_lemma}, together with \eqref{gCAlphaNorm} imply
\begin{equation}\label{gCEpsNorm}
\underset{{\substack{
            \xi, \mu  \in \Gam \\
            \xi \neq \mu } }}{\sup} \sum_{j=1}^{\INF}\frac{ j \lvert g_{-j}(\xi) - g_{-j}(\mu) \rvert}{ \ds \lvert \xi - \mu \rvert^{ \eps}}
\leq \underset{{\substack{
            \xi, \mu  \in \Gam \\
            \xi \neq \mu } }}{\sup} \, \frac{ \ds \lVert g(\xi,\cdotp) - g(\mu,\cdotp) \rVert_{C^{1,\alpha}}}{ \ds \lvert \xi - \mu \rvert^{ \eps}} < \INF.
\end{equation}
By combining the estimates \eqref{gBernProp} and \eqref{gCEpsNorm} we showed that $\Pminus (g) \in C^{\eps}(\Gam ; l_{1})$. This proves part (i).

Now let $g\in C^{\eps} \left ( \Gam ; C^{1,\alpha}(\sph)\right) \cap C^0\left ( \Gam;C^{2,\alpha}(\sph)\right)$. Since $g\in C^0\left ( \Gam;C^{2,\alpha}(\sph)\right)$, then
$$\underset{\xi \in \Gam}{\sup} \lVert g(\xi,\cdotp) \rVert_{C^{2,\alpha}}  < \INF.$$
Lemma \ref{bernstein_lemma} applied to $g(\xi,\cdotp) \in C^{2,\alpha}$ for $\xi\in\Gamma$ yields
\begin{equation}\label{gBernPropY}
\sup_{w\in \Gam} \sum_{j=1}^{\INF} j^{2} \,\lvert g_{-j}(w) \rvert \leq  \lVert g(\xi,\cdotp) \rVert_{C^{2,\alpha}}.
\end{equation}
This shows that $\Pminus (g)\in l^{1,2}_{\INF}(\Gamma)$.
Now \eqref{gCEpsNorm} yields $\Pminus (g) \in Y_{\eps}$.

By triangle inequality in \eqref{PstarOp}, we have
$\bg\in C^{1,\alpha}(\Omega;l^1)\cap C^\alpha(\overline\Omega;l^1)$ yields
\begin{align*}
\sup_{\xi\in\ol\OM}\|\bg(\xi)\|_{l^1}+
\underset{{\substack{
            \xi, \mu  \in \ol\OM\\
            \xi \neq \mu } }}\sup\frac{\lVert \bg (\xi) - \bg(\mu) \rVert_{l^1}}{\lvert \xi - \mu \rvert^\alpha} < \INF.
\end{align*}
For $\xi\in\OM$, and $r>0$ with $\ol{B(\xi;r)}\subset \OM$, there is an $M_{\xi,r} > 0$ with
\begin{align*}
\sup_{\xi\in\ol{B(\xi;r)}}\|\nabla\bg(\xi)\|_{l^1}+
\underset{{\substack{
            \mu  \in \ol{B(\xi;r)}\\
            \xi \neq \mu } }}\sup\frac{\lVert\nabla \bg (\xi) - \nabla\bg(\mu) \rVert_{l^1}}{\lvert \xi - \mu \rvert^\alpha} \leq M_{\xi,r}.
\end{align*}
These proves part(iii).

\end{proof}

The following result refines the mapping properties of the operator $G$ in \eqref{Gdefn}, when restricted to the subspace $Y_{\eps}$.
\begin{prop}\label{GCHpropext}
Let $Y_{\eps}$ be the space defined in \eqref{YGamdefn}. Then
\begin{align*}
&(i) \; G : Y_{\eps} \longrightarrow C^{\eps }\left( \ol{\OM}; l^{1} \right) \cap C^{1}\left( \OM; l^{1} \right),\\
&(ii) \;\HT_0 : Y_{\eps} \longrightarrow C^{ \eps }\left( \Gam ; l^{1} \right).
\end{align*}
\end{prop}

\begin{proof} (i) Let $ \xi, \xi_{0} \in\ol \OM$ and $\bg \in Y_{\eps} .$ Using the parametrization $ w(\fii) = \xi \; + \; l_{\xi}(\fii)  e^{i \fii}$, where $\,l_{\xi}(\fii) = |\xi - w(\fii)| $, we obtain as in the proof in Proposition \ref{Gprop} that
\begin{equation}\nn
 (G \bg)_{-n}(\xi) =  \frac{2}{\pi } \int_{0}^{2\pi}
\sum_{j=1}^{\INF}  g_{-n-2j}(\xi + l_{\xi}(\fii) \, e^{i \fii}) \; e^{-2ij\fii} d\fii,
\end{equation} is well defines for $\xi \in \ol \OM$.

Since $\bg\in Y_{\eps}$, we have
\begin{align*}
\kappa:=\underset{{\substack{
            \xi, \mu  \in \Gam \\
            \xi \neq \mu } }}{\sup} \sum_{j=1}^{\INF} j\frac{\lvert g_{-j}(\xi) - g_{-j}(\mu) \rvert}{|\xi - \mu|^\eps} < \INF
\end{align*}

We estimate
\begin{align*}
\sum_{n=0}^{\INF} &\lvert (G \, \bg)_{-n}(\xi) - (G \, \bg)_{-n}(\xi_{0}) \rvert\\
& \leq
\frac{2}{\pi} \sum_{n=0}^{\INF}  \sum_{j=1}^{\INF}\int_{0}^{2\pi} \left |g_{-n-2j}(\xi + l_{\xi}(\fii) \, e^{i \fii}) - g_{-n-2j}(\xi_{0} + l_{\xi_{0}}(\fii) \, e^{i \fii}) \,\right | d\fii\\
& \leq
\frac{2}{\pi} \sum_{j=1}^{\INF} \int_{0}^{2\pi} j \left |g_{-j}( \xi + l_{\xi}(\fii) \, e^{i \fii} ) -
g_{-j}( \xi_{0} + l_{\xi_{0}}(\fii) \, e^{i \fii}) \,\right | d\fii\\
& \leq
\frac{2\kappa}{\pi}\int_{0}^{2\pi} \left | (\xi - \xi_{0})  + \left [ l_{\xi}(\fii) - l_{\xi_{0}}(\fii) \right ]  \, e^{i \fii} \right |^{\eps} d\fii, \\
& \leq \frac{2\kappa}{\pi} \int_{0}^{2\pi} \; \left ( 2 \left| \xi - \xi_{0} \right |^{\eps} + \left| l_{\xi}(\fii) - l_{\xi_{0}}(\fii) \right|^{\eps}\right )  d\fii,\\
&\leq (8\kappa+ 4\kappa L^\eps) \left| \xi - \xi_{0} \right |^{\eps}.
\end{align*} In the second inequality we used Lemma \ref{seqresult} part (ii), in the third inequality we used $|a+b|^{\epsilon} \leq 2|a|^{\epsilon} + |b|^{\epsilon}$,
and in the fourth inequality we used \eqref{lipshitz_constant}. This shows $G \bv \in C^{\eps}(\ol\OM; l^{1})$.

We will show next that $G\bg\in C^{1}(\OM ; l^{1})$. Suffices to carry
the estimates in the neighborhood $\ol{B(\xi_0,r_0)}\subset\OM$ of an arbitrary point $\xi_0 \in \OM $, where $r_0=\dist(\xi_0,\Gamma)/2>0$. Recall the estimate \eqref{gradientEstimate} where $c = 2/r_{0}^2$.

For each $n\leq 0$, we have
\begin{align*}
\nabla_{\xi} (G \, \bg)_{n}(\xi)
&= \frac{1}{\pi i} \int_{\Gam } \nabla_{\xi} \left \{ \frac{dw}{w-\xi}-\frac{d \ol{w}}{\ol{w-\xi}} \right \} \sum_{j=1}^{\INF} \ds g_{n-2j}(w) \left ( \frac{\overline{w-\xi}}{w-\xi} \right) ^{j}  \\ &\qquad \quad {}  + \frac{1}{\pi i}\int_{\Gam }  \left \{ \frac{dw}{w-\xi}-\frac{d \ol{w}}{\ol{w-\xi}} \right \} \sum_{j=1}^{\INF} \ds g_{n-2j}(w)  \nabla_{\xi} \left ( \frac{\ol{w-\xi}}{w-\xi} \right) ^{j},
\end{align*}
which estimate using \eqref{gradientEstimate} by
\begin{align*}
\sum_{n=0}^{\INF} \left | \nabla_{\xi} (G \, \bg)_{-n}(\xi) \right |
 \leq& \frac{ c}{\pi} \sum_{n=0}^{\INF} \int_{\Gam } \sum_{j=1}^{\infty} \ds \left | g_{-n-2j}(w) \right | dw   \\
 &+  \frac{ c}{\pi} \sum_{n=0}^{\INF} \int_{\Gam } \sum_{j=1}^{\INF}  j \left | g_{-n-2j}(w) \right | dw.
\end{align*}
By Lebesgue Dominated Convergence Theorem, Lemma \ref{seqresult}, and $\bg\in l^{1,2}_{\INF}(\Gamma)$, the right hand side above is finite.

To prove part (ii) we note that $Y_{\eps} \subset C^\epsilon(\Gamma;l^1)$ so that the Sokhotzki-Plemelj limit in \eqref{plemelj_l1} holds. The result follows from Definition \ref{hilbertT_definition} of $\HT_0$ and part(i) above.
\end{proof}

\begin{cor}\label{CorTh2NecY}
Let $Y_{\eps}$ be the space defined in \eqref{YGamdefn} and $\bg \in Y_{\eps}$ satisfying
\begin{align} \label{NecSufEqonY}
 (I+i\HT_0) \bg=0.
 \end{align} Define $\bv = \langle v_0, v_{-1}, v_{-2}, ...\rangle $ by the Cauchy Integral formula \eqref{vnDefn}.
 Then $\bv \in C^{1,\eps}(\OM;l^{1})$ extends continuously to a map in $ C^{\eps}(\ol \OM;l^{1})$. Moreover, $\bv$ is A-analytic
 and $\bv \lvert_{\Gam} = \bg$, in the sense
 $$\lim_{\OM\ni z \to z_{0}\in \Gam}
\left\| \bv (z) - \bg(z_{0})\right\|_{l_1} =0,$$
and, for $\Pstar$ in \eqref{PstarOp}, we have
$$\lim_{\OM\ni z \to z_{0}\in \Gam}
\Pstar( \bv) (z) = \Pstar( \bg)(z_{0}).$$
\end{cor}
\begin{proof}
Since $\bg \in Y_{\eps}$, by Proposition \ref{GCHpropext} we have $ \bv\in C^{\eps}(\ol\OM;l^1)\cap C^1(\OM;l^1)$.
By summing \eqref{del_v_n} and \eqref{dbar_v_n} we obtain $$ \ol{\del} v_{n} + \del v_{n-2} = 0, \quad n=0,-1,-2,...,$$ and so $\bv$ is A-analytic. Next we will show that $\bv \lvert_{\Gam} = \bg$. Let $z \in \OM$ and $z_{0}\in \Gam$. Then
\begin{align*}
\| \bv (z)& - \bg(z_{0})\|_{l_1}
= \left\| \frac{1}{2}(G\bg)(z) +(C \bg)(z) - \bg(z_{0})\right\|_{l_1} \\
&= \left\| \frac{1}{2}\left( (G\bg)(z) - (G\bg)(z_{0}) \right ) +
\left ( (C \bg)(z) - \frac{1}{2} \bg(z_{0}) - \frac{1}{2} S\bg(z_0) \right )  \right\|_{l_1}
\end{align*}
In the equality above we use the fact that $\bg $ satisfy  \eqref{NecSufEqonY}.
From Proposition \ref{Plemeljpropext} and Proposition \ref{GCHpropext} part (i) we have
\begin{align*}
\lim_{\OM\ni z \to z_{0}\in \Gam} \left\|
(C \bg)(z) - \frac{1}{2} \bg(z_{0}) - \frac{1}{2} S\bg(z_0)   \right\|_{l_1}  = 0,\\
\lim_{\OM\ni z \to z_{0}\in \Gam} \left\|  (G\bg)(z) - (G\bg)(z_{0})  \right\|_{l_1} = 0,
\end{align*} and so
\begin{equation}\label{trace_lim}
\lim_{\OM\ni z \to z_{0}\in \Gam}
\left\| \bv (z) - \bg(z_{0})\right\|_{l_1} = 0,\end{equation}
i.e. $\bv \lvert_{\Gam} = \bg$.

Since $ \bv\in C^{\eps}(\ol\OM;l^1)\cap C^1(\OM;l^1)$, it follows from Proposition \ref{gCepsgbarXY} part (iii) that
$\Pstar v(z) \in C^{1,\eps} (\OM \times \sph) \cap C^{\eps}( \ol \OM \times \sph )$.  The triangle inequality yields
\begin{align*}
\left|[\Pstar\bv](z,\theta)-[\Pstar\bg](z_0,\theta)\right| \leq \left\| \bv (z) - \bg(z_{0})\right\|_{l_1},
\end{align*}
and the result follows from \eqref{trace_lim}.
\end{proof}

\begin{lemma}\label{DTauDisctsVar}
Let $\OM$ be a (convex) domain with $C^2$-boundary $\Gamma$ with a strictly positive curvature lower bound $\delta>0$.
Let $\tau (z,\tta)$ be as in \eqref{chord} for $(z,\tta)\in \ol\OM \times \sph,$ then the angular derivative $\del_\fii\tau (z,\tta)$ has a jump discontinuity across the variety $Z$ as defined by
\begin{align}\label{variety}
Z:=\{(z,\theta)\in\Gamma\times\sph:\; \bn(z)\cdot \theta=0 \}.\end{align}
\end{lemma}
\begin{proof}
Let $z_{0}\in \Gam$ be fixed and let $\theta_0:=\bn(z_0)^\perp$ with $\fii_0=\arg(\theta_0)$.

Let $\tilde\tau(z_0,\theta)$ be the length of the chord corresponding to the osculating circle at $z_0$ of radius $R_0$ and let $\fii=\arg(\theta)$.
Let $\tau (z_{0},\tta)$ be the length of the chord from $z_0$ to the boundary in the $\tta$ direction as defined in \eqref{chord}.

Consider a local parametrization  $t\mapsto ( t,y(t) )$ of the boundary near $z_0=(0,y(0))\in \Gamma$, with $y(0)=y'(0)=0$. Then the curvature of the boundary at $z_0$ is $k(0)=y''(0)$, and, by the Taylor series expansion,
$$y(t) =\frac{\kappa(0)t^2}{2}+r(t)t^2,$$ for some $r(t)$ with
$\lim_{t\to 0}r(t) = 0$.

The equation of the line passing through $z_{0}$ and making an angle $\fii-\fii_0$ with the positive $t$ axis is $(t,\tan(\fii-\fii_0) t)$.
The point of intersection of this line with $\Gam$ gives $t = \ds \frac{2 \tan (\fii-\fii_0)}{\kappa(0)+2r(t)}$.
Thus, \begin{align*}
\tau(z_0,\theta) = t \sec (\fii-\fii_0) = \ds \frac{\sin (\fii-\fii_0)}{\cos^{2}(\fii-\fii_0)} \frac{2}{\kappa(0)+2r(t)}\\ \leq \frac{2c_{1}}{\kappa(0)} \ds \frac{\sin (\fii-\fii_0)}{\cos^{2}(\fii-\fii_0)}
\leq \frac{2c_{1}R_0|\sin(\fii-\fii_0)|}{\cos^{2}(\fii-\fii_0)}.\end{align*}

\begin{figure}[ht]
\centering

\begin{tikzpicture}[scale=1.5,cap=round,>=latex]

 \draw[name path=ellipse, thick,blue] (0cm,1cm) ellipse (1.6cm and 2cm);
 \draw[name path=circle, thick,red] (0cm,0cm) circle(1cm);

 \coordinate[label=above:$\Omega$] (OM) at(100:2.6cm);

   \draw[->] (0,-1cm) -- (1.8cm,-1cm);
 \coordinate[label=right:$t$] (t) at (1.81cm,-1cm);

  \draw[->] (270:1cm) -- (90:1.6cm);

  \filldraw[black] (270:1cm) circle(1.2pt);
  \coordinate[label=right:$z_{0}$] (z) at(268:1.13cm);
  \coordinate (z0) at(270:1cm);

  \draw[->] (270:1cm) -- (270:1.6cm);
   \coordinate[label=below:$n(z_{0})$] (normal) at (270:1.62cm);

   \draw[->] (0,-1cm) -- (45:2.8cm);
 \coordinate[label=above:$\theta$] (theta1) at (45:2.76cm);

   \draw [green,thick] (0.6,-1) arc (0:40:0.875cm);
   \draw [color=black](0.72cm,-0.70cm) node[rotate=-65, scale=0.75] {$\phi - \phi_{0}$};

 \filldraw[black] (22:1cm) circle(1.2pt);
 \coordinate (x1) at (22:1cm);

 \filldraw[black] (41:2.1cm) circle(1.2pt);
  \coordinate[label=right:$x_{\theta}^{+}$] (x2) at (41:2.12cm);
  \coordinate(xtta) at (41:2.1cm);

  \tikzset{
    position label/.style={
       below = 3pt,
       text height = 1.5ex,
       text depth = 1ex
    },
   brace/.style={
     decoration={brace,mirror},
     decorate
   }
}
\draw [brace,decoration={raise=0.5ex}] (z0.north) -- (xtta.north) node [position label, pos=0.55, rotate = 50, scale=0.8] {$\tau(z_{0},\tta)$};

  \tikzset{
    position label/.style={
       above = 2.5pt,
       text height = 1.1ex,
       text depth = 1ex
    },
   brace/.style={
     decoration={brace,mirror},
     decorate
   }
}
\draw [brace,decoration={raise=0.5ex}] (x1.north) -- (z0.north)  node [position label, pos=0.45, rotate = 52,scale=0.8] {$\tilde{\tau}(z_{0},\tta)$};

\end{tikzpicture}
\caption{} \label{fig2:Lemma 2}
\end{figure}

From the geometry of the osculating circle (see Figure \ref{fig2:Lemma 2}), we have
\begin{align}\label{osculatingEst}
\tilde\tau(z_0,\theta)=2R_0|\sin(\fii-\fii_0)|\leq \frac{2}{\delta}|\fii-\fii_0|,
\end{align}and so there is a constant $C>0$, such that, for all $(z_0,\theta)\in\Gamma\times\sph$,
\begin{align}
\tau(z_0,\theta)\leq C\tilde\tau(z_0,\theta).
\end{align}
A derivative in $\fii$ at $\fii_0$ in the equality in \eqref{osculatingEst} also yields the jump value of $4R_0$, as the direction $\theta$ crosses the tangent direction from outgoing to incoming.
\end{proof}

In order for the integral in \eqref{AttRT} to inherit the regularity of $f$ it is then necessary for $f$ to vanish at the boundary. The following proposition makes this statement precise.

\begin{cor}\label{nonempty}Let $\OM$ be a (convex) domain with $C^2$-boundary $\Gamma$ with a strictly positive curvature lower bound $\delta>0$.
If $f \in C^{1,\alpha}_0(\ol\OM)$, then $Rf\cap C^{1,\alpha}(\Gam\times\sph) \neq\emptyset$.
\end{cor}

\begin{proof}
For every $(z,\theta)\in\Gamma\times\sph$,  let us define
\begin{equation}\label{sampleData}
g(z,\theta)=\left\{\begin{array}{ll}
\int_{-\tau(z,\theta)}^0 f(z+t\tta)dt,& \bn(z)\cdot \theta>0,\\
0,&\bn(z)\cdot \theta\leq 0,
\end{array}\right.
\end{equation}where $\bn(z)$ is the unit outer normal at $z\in\Gam$.
Since
\begin{equation}\nonumber
g(z^+_\theta,\theta)=\left\{\begin{array}{ll}
\int_{-\tau(z^+_\theta,\theta)}^0 f(z^+_\theta+t\tta)dt,& \bn(z^+_\theta)\cdot \theta>0,\\
0,&\bn(z^+_\theta)\cdot \theta\leq 0,
\end{array}\right.
\end{equation} and $g(z^-_\theta,\theta)=0,$ condition \eqref{Radon_definition}
is satisfied with $a \equiv0 $ to show that $g\in Rf$.
We will show next that $g\in C^{1,\alpha}(\Gam\times\sph)$.
Let $\del$ be the partial derivative with respect to one of the spacial or angular variable.
At points $(z_0,\theta_0)\in (\Gamma\times\sph)\setminus Z$, differentiation in \eqref{sampleData} together with $f|_{\Gamma}=0$ yield
\begin{equation}\label{sampleDataDerv}
\del g(z_0,\theta_0)=\left\{\begin{array}{ll}
\int_{-\tau(z_0,\theta_0)}^0 \del f(z_0+t\tta_0)dt,& \bn(z_0)\cdot \theta_0>0,\\
0,&\bn(z_0)\cdot \theta_0< 0,
\end{array}\right.
\end{equation}

Since $\del f\in C^\alpha(\ol\OM)$, it remains to show that $\del f$ extends $C^\alpha$ across the variety $Z$.
We first consider the case for a fixed $z_0\in\Gamma$ and study the dependence of $\del g$ in $\theta$ near the tangential direction $\theta_0:=\bn(z_0)^\perp$.
The other case, studying the dependence of $\del g$ as $z \in \Gam$ approach $z_{1}$ along $\Gam$  for a fixed $\tta_{1} \in \sph$ with $(z_{1}, \tta_{1}) \in \Gam \times \sph$ reduces to the first case.

For this we first analyze the speed of convergence of $\tau(z_0,\theta)\to 0$ as $\theta\to\theta_0$. Let $\tilde\tau(z_0,\theta)$ be the length of the chord corresponding to the osculating circle at $z_0$ of radius $R_0$.
From Lemma \ref{DTauDisctsVar}, we have that there is a constant $C>0$, such that, for all $(z_0,\theta)\in\Gamma\times\sph$,
\begin{align}
\tau(z_0,\theta)\leq C\tilde\tau(z_0,\theta).
\end{align}
From the geometry of the osculating circle (see Figure \ref{fig2:Lemma 2}), we have
\begin{align}\label{CorOsculatingEst}
\tilde\tau(z_0,\theta)=2R_0|\sin(\fii-\fii_0)|\leq \frac{2}{\delta}|\fii-\fii_0|,
\end{align}
where $\fii=\arg(\theta)$ and $\fii_0=\arg(\theta_0)$.
A derivative in $\fii$ at $\fii_0$ in the equality in \eqref{CorOsculatingEst} also yields the jump value of $4R_0$, as the direction $\theta$ crosses the tangent direction from outgoing to incoming direction.
Since $\lim_{\theta\to\theta_0}\tau(z_0, \theta)=0$, the formula \eqref{sampleData} shows that $g\in C^1(\Gamma\times\sph)$.
To prove that $\del g$ is $\alpha$-H\"{older} continuous, we estimate using \eqref{osculatingEst}
\begin{align}
|\del g(z_0,\theta)|\leq \|\nabla f\|_{\infty}\tau(z_0,\theta)\leq  \|\nabla f\|_{\infty}C\tilde\tau(z_0,\theta)|\leq \tilde{C}|\fii-\fii_0|,
\end{align}for some constant dependent on the $\sup$-norm of the $|\nabla f|$ and the minimum curvature $\delta$.
\begin{align*}
&|g(z_0,\theta) - g(z_0,\theta_{0})|\\
&= \left| \int_{-\tau(z_0,\theta_0)}^0 (\del f(z_0+t\tta_0) - \del f(z_0+t\tta_0) )dt + \int_{-\tau(z_0,\theta_0)}^{-\tau(z_0,\theta)} \del f(z_0+t\tta_0) dt \right| \\ \nn
&\leq C_{1} |\fii-\fii_0|^{\alpha} \tau(z_0,\theta) + \|\nabla f\|_{\infty} |\tau(z_0,\theta) - \tau(z_0,\theta_{0}) | \\ \nn
&\leq \tilde{C} |\fii-\fii_0|^{\alpha} + \|\nabla f\|_{\infty} C_{2} |\fii-\fii_0| \leq C |\fii-\fii_0|^{\alpha}
\end{align*}
Therefore $g\in Rf\cap C^{1,\alpha}(\Gam\times\sph)$.
\end{proof}

One of our main results establishes necessary and sufficient conditions for a sufficiently smooth function on $\Gam \times \sph$ to be the Radon data of some sufficiently smooth source as follows.

\begin{theorem}[Range characterization for Radon transform]\label{Th3}
Let $\OM\subset\mathbb{R}^2$ be a domain with $C^2$ boundary $\Gamma$ of strictly positive curvature, and $\alpha > 1/2$.

(i) Let $f \in C_0^{1,\alpha}(\ol\OM)$ be real valued, and $g\in Rf\cap C^{\alpha}(\Gam;C^{1,\alpha}(\sph))$. Then $\Pminus (g)$ as defined in \eqref{Pminusplus} solves
\begin{align}\label{Th3Cond1}
[I+i\HT_0]\Pminus (g)=0,
\end{align}where $\HT_0$ is the Hilbert transform in \eqref{hilbertT}.

(ii) Let $g\in C^{\alpha} \left(\Gam; C^{1,\alpha}(\sph) \right)\cap C^0(\Gam;C^{2,\alpha}(\sph))$ be real valued and such that $\Pminus (g)$ satisfies \eqref{Th3Cond1}. Then there exists a real valued $f \in C^\alpha(\OM)\cap L^1(\OM)$, and such that $g\in Rf$.
\end{theorem}

\begin{proof}
(i) By Corollary \ref{nonempty}, we note first that $Rf\cap C^{\alpha} \left(\Gam; C^{1,\alpha}(\sph) \right)\supset Rf\cap C^{1,\alpha}(\Gam\times\sph)\neq\emptyset$. Since $g\in C^{\alpha} \left( \Gam ;C^{1,\alpha}(\sph) \right)$, by Proposition \ref{gCepsgbarXY} part (i), we have that $\Pminus (g) \in l^{1,1}_{\INF}(\Gamma)\cap C^{\alpha}(\Gamma;l^1)$. Now the necessity in Theorem \ref{NecSuf} applies to yield
$(I+i\HT_0)\Pminus (g)=0$.

Next we prove the sufficiency of \eqref{Th3Cond1} in part (ii). 

Since $g \in C^{\alpha} \left (\Gam ; C^{2,\alpha}(\sph) \right) \cap C^{0} \left (\Gam; C^{2,\alpha}(\sph)\right)$, it follows from the Proposition \ref{gCepsgbarXY} part (ii) that $\bg := \Pminus (g)\in Y_{\eps}$. For each $z \in \OM$, construct the vector valued function
$\bv = \langle v_0, v_{-1}, v_{-2}, ...\rangle $ by the Cauchy Integral formula \eqref{vnDefn}:
\[v_{n} (z)= \frac{1}{2}(G \bg)_{n}(z)+ (C\bg)_{n}(z), n=0,-1,-2...\]

By Corollary \ref{CorTh2NecY}, $\bv \in C^{1,\eps}(\OM;l^{1}) \cap C^{\eps}(\ol \OM;l^{1})$ is A-analytic, in particular for each $n=0,-1,-2,...$, we have $$\ol{\del} v_{n} + \del v_{n-2} = 0.$$

Using $v_{-1}\in C^{1,\alpha}(\OM)$ we define the H\"{o}lder continuous function $f\in C^\alpha(\OM)$ by
\begin{equation}\label{fDefnTh3}
f(z) := 2 \re \left ( \del v_{-1}(z) \right ),\quad z\in\OM,
\end{equation}
and show that $f$ integrates along any line and that $g\in Rf$.

Since $\bv \in C^{1,\eps}(\OM;l^{1}) \cap C^{\eps}(\ol \OM;l^{1})$, it follows from the Proposition \ref{gCepsgbarXY} part (iii)
that $$v(z,\tta) := \Pstar (\bv(z)) \in C^{1,\alpha}(\OM \times \sph) \cap C^{\alpha}(\ol \OM \times \sph).$$
Also from Corollary \ref{CorTh2NecY}, $\bv \lvert_{\Gam} = \bg $ and $\ds \lim_{\OM\ni z \to z_{0}\in \Gam}
v(z,\tta) = \Pstar \bg(z_{0})  $. Now using the fact that $g$ is real valued yields $\ds \lim_{\OM\ni z \to z_{0}\in \Gam}
v(z,\tta) = g(z_{0},\tta )$ i.e $\ds v \lvert_{\Gam\times \sph} = g$.

Using $\tta\cdot\nabla v= e^{-i\fii} \ol{\del v} + e^{i\fii} (\del v)$, we obtain
\begin{align*}
\tta\cdot\nabla v(z, \tta)
&=
2 \re \left ( \del v_{-1}(z) \right ) + 2\re\left(\sum_{n=0}^\INF ( \dba v_{-n}(z) + {\del} v_{-n-2}(z))e^{-in\fii}\right) \\ \nn
&= 2 \re \left ( \del v_{-1}(z) \right ) =f(z).
\end{align*}
By integrating $f(z)=\tta \, \cdotp \, \nabla v(z,\tta),$
we obtain
\begin{align*}
\int_{\tau_{-}(z,\tta)}^{\tau_{+}(z,\tta)} f(z + s \tta) ds
&=\underset{{\substack{
            t_{1} \to -\tau_{-}(z,\tta) \\
            t_{2} \to \tau_{+}(z,\tta) } }}{\lim}\int_{t_{1}}^{t_{2}} f(z + s \tta) ds
            \\
&=\underset{{\substack{
            t_{1} \to -\tau_{-}(z,\tta) \\
            t_{2} \to \tau_{+}(z,\tta) } }}{\lim}
 \left [ v(z + t_{2} \tta ,\tta) -v(z + t_{1} \tta ,\tta) \right] \\
&=
\ds g \left( z\, + \,\tau_{+}(z,\tta)\;\tta \,,\; \tta  \right )
- \ds g \left( z\, - \,\tau_{-}(z,\tta)\;\tta \,,\; \tta  \right )
\end{align*}
This shows that $f$ integrates along any arbitrary line, in particular $f\in L^1(\OM)$, and that $g\in Rf$.
\end{proof}

\section{Range characterization for the attenuated Radon transform}
In this section we consider the attenuated case, where $a\not\equiv0$ is a real valued map. 
The method of proof is based on the reduction to the non-attenuated case.
Since $e^{-Da}$ in \eqref{divbeam} is an integrating factor, the equation \eqref{eq:TEq2} can be rewritten in the advection form similar to \eqref{eq:TEql} as
\begin{align}\label{eq:ReTEq2}
\ds \tta\cdot\nabla \left( e^{-Da(z,\tta)} u(z,\tta) \right ) = f(z)  e^{-Da(z,\tta)}.
\end{align}
However, the right hand side is now angularly dependent with nonzero positive and negative modes, and one cannot use the A-analytic equations \eqref{Aanalytic} directly. The key idea in the reduction of the attenuated to the non-attenuated case is to alter the integrating factor in such a way that all the negative Fourier modes vanish. Let $h$ be defined in $\OM\times\sph$ by

\begin{align}\label{hDefn}
h(z,\tta) := Da(z,\theta) -\frac{1}{2} \left( I - i H \right) Ra(z\cdot \tta^{\perp},\tta),
\end{align}where $Ra(s,\tta) = \ds \int_{-\INF}^{\INF} a\left( s \tta^{\perp} +t \tta \right)dt$ is the Radon transform of the attenuation, and
the classical Hilbert transform $H h(s) = \ds \frac{1}{\pi} \int_{-\INF}^{\INF} \frac{h(t)}{s-t}dt $ is taken in the first variable and evaluated at $s = z \cdotp \tta^{\perp}$.  Since we altered $Da$ by a function which is constant in $x$ in the direction of $\theta$, the function $e^{-h}$ still remains an integrating factor for \eqref{eq:TEq2}. The integrating factor in \eqref{hDefn} was first considered in the work of Natterrer \cite{naterrerBook}; see also \cite{finch}, and \cite{bomanStromberg} for elegant arguments that show how $h$ extends from $\sph$ inside the disk as an analytic map. Since $e^{\pm h}$ are also extension of analytic functions in the disk they still have vanishing negative modes.

\begin{prop}\label{hprop} Let $a \in C_{0}^{1,\alpha}(\ol\OM), \alpha>1/2$, and $h$ be defined in \eqref{hDefn}.
Then $h \in C^{1,\alpha}(\ol \OM\times\sph)$.
\end{prop}
\begin{proof}
 Since $a \in C_0^{1,\alpha}(\ol\OM)$, we use the proof of Corollary \ref{nonempty} applied to $a$ to conclude $Da \in C^{1,\alpha}(\ol\OM\times\sph)$ and also $Ra \in C^{1,\alpha}(\mathbb{R}\times\sph)$. The Hilbert Transform in the linear variable preserve the smoothness class to yeild  $ H  Ra  \in C^{1,\alpha}(\ol\OM\times\sph)$ and thus $h \in
C^{1,\alpha}(\ol\OM\times\sph) $.
\end{proof}

Consider the Fourier expansions of $e^{-h(z,\tta)}$ and $e^{h(z,\tta)}$ \begin{align}\label{ehEq}
  e^{-h(z,\tta)} = \sum_{k=0}^{\INF} \alpha_{k}(z) e^{ik\fii}, \quad e^{h(z,\tta)} = \sum_{k=0}^{\INF} \beta_{k}(z) e^{ik\fii}, \quad (z, \tta) \in \ol\OM \times \sph
\end{align} where $h \in C^{1,\alpha}(\Gam\times\sph)$ is as defined in \eqref{hDefn}. Since $e^{-h}e^{h} = 1 $ the Fourier modes $\alpha_{k}, \beta_{k}, k \geq 0$ satisfy
\begin{align}\label{alphabetaEq}
\alpha_{0} \beta_{0} = 1, \qquad \ds \sum_{m = 0}^{k} \alpha_{m} \beta_{k-m} = 0, \; k \geq 1.
\end{align}
The following mapping property is used in defining Hilbert Transform associated with attenuated Radon Transform.
Recall the operator $\Pplus$ in \eqref{Pminusplus},
$e^{h}$ be as in \eqref{ehEq},
and $Y_{\alpha}$ in \eqref{YGamdefn} with $\eps = \alpha$.
\begin{prop}\label{ehConvProp}
Let $a \in C_{0}^{1,\alpha}(\ol\OM)$ with $\alpha>1/2$.  Then $\Pplus (e^{\pm h}) \in C^{\alpha}(\ol\OM;l^{1})$. Moreover
\begin{align*}
&(i)\; \Pplus (e^{h}) \ast_{n} (\cdot): C^{\alpha} (\ol\OM ; l_{\INF}) \to C^{\alpha} (\ol\OM ; l_{\INF}); \\
&(ii)\; \Pplus (e^{h}) \ast_{n}(\cdot): C^{\alpha} (\ol\OM ; l_{1}) \to C^{\alpha} (\ol\OM ; l_{1}); \\
&(iii)\; \Pplus (e^{h}) \ast_{n}(\cdot): Y_{\alpha} \to Y_{\alpha},
\end{align*} where $\ast_{n}$ denotes the convolution operator on sequences.

\end{prop}
\begin{proof}
Since $a \in C_0^{1,\alpha}(\ol\OM)$, it follows from Proposition \ref{hprop} that $e^{\pm h} \in C^{1,\alpha}(\ol\OM \times \sph) \subset C^{\alpha}(\ol\OM;C^{\alpha}(\sph))$.
Then
\begin{equation} \label{ehCAlphaNorm}
\ds \underset{z \in \ol\OM}{\sup} \, \lVert e^{h(\xi,\cdotp)} \rVert_{C^{\alpha}(\sph)}
+ \underset{{\substack{
            \xi, \mu  \in \ol \OM \\
            \xi \neq \mu } }}{\sup} \, \frac{ \ds \lVert e^{h(\xi,\cdotp)} - e^{h(\mu,\cdotp)} \rVert_{C^{\alpha}(\sph)}}{ \ds \lvert \xi - \mu \rvert^{\alpha}} < \INF.
\end{equation}
Let $\Pplus (e^{h}) := \langle \beta_{0},\beta_{1},\beta_{2}, \cdots \rangle$. Then
\begin{equation}\label{betaLNorm}
\underset{\xi \in \ol \OM}{\sup} \, \sum_{k=1}^{\INF} \lvert \beta_{k}(\xi) \rvert\leq
 \underset{\xi \in \ol \OM}{\sup} \, \lVert e^{h(\xi,\cdotp)} \rVert_{C^{\alpha}(\sph)}<\infty.
\end{equation}

Another application of Lemma \ref{bernstein_lemma} together with \eqref{ehCAlphaNorm} imply
\begin{equation}\label{betaCalpNorm}
\underset{{\substack{
            \xi, \mu  \in \ol \OM \\
            \xi \neq \mu } }}{\sup} \sum_{k=1}^{\INF}\frac{\lvert \beta_{k}(\xi) - \beta_{k}(\mu) \rvert}{ \ds \lvert \xi - \mu \rvert^{\alpha}}
\leq \underset{{\substack{
            \xi, \mu  \in \ol \OM \\
            \xi \neq \mu } }}{\sup} \, \frac{ \ds \lVert e^{h(\xi,\cdotp)} - e^{h(\mu,\cdotp)} \rVert_{C^{\alpha}}}{ \ds \lvert \xi - \mu \rvert^{\alpha}} < \INF.
\end{equation}
By combining the estimates \eqref{betaLNorm} and \eqref{betaCalpNorm} we showed that $\Pplus (e^{h}) \in C^{\alpha}(\ol \OM ; l_{1})$. A similar estimate shows $\Pplus (e^{-h}) \in C^{ \alpha}(\ol \OM ; l_{1})$.

Next we prove part (i). Let $\bg \in C^{\alpha} (\ol\OM ; l_{\INF}) $, and $\bv: = \Pplus (e^h)\ast_{n} \bg$ given by
$$v_{n} = \sum_{k=0}^{\INF}\beta_{k}g_{n-k},\quad n \leq 0,$$
where $\beta_{k}$ are the Fourier coefficients of $e^{h}$, as in \eqref{ehEq}.
Since $\bg \in C^{\alpha} (\ol \OM ; l_{\INF})$ and $\Pplus(e^h)\in C^\alpha(\ol\OM;l^1) $, we have
\begin{align}\label{Estginf}
c_{1}:= \sup_{n \leq 0} \sup_{\xi \in \ol\OM} |g_{n}(\xi)| < \INF,
\;
\kappa_{1}:= \sup_{n \leq 0} \underset{{\substack{
            \xi,\eta \in \ol \OM \\
            \xi \neq \eta } }}{\sup} \, \frac{ \ds \lvert g_{n}(\xi) - g_{n}(\eta) \rvert}{ \ds \lvert \xi - \eta \rvert^{\alpha}} < \INF,
\end{align}
and
\begin{align}\label{EstEhl1}
c_{2}:= \sup_{\xi \in \ol \OM} \sum_{k=0}^{\INF}|\beta_{k}(\xi)| < \INF, \;
\kappa_{2}:= \underset{{\substack{
            \xi,\eta \in \ol \OM \\
            \xi \neq \eta } }}{\sup} \sum_{k=0}^{\INF}\, \frac{ \ds \lvert \beta_{k}(\xi) - \beta_{k}(\eta) \rvert}{ \ds \lvert \xi - \eta \rvert^{\alpha}}< \INF.
\end{align}

By taking the supremum in $\xi\in\overline\OM$, for each $n \leq 0$, in
\begin{align*}
|v_{n}(\xi)| \leq  \sum_{k=0}^{\INF} \left |\beta_{k}(\xi)g_{n-k}(\xi) \right | \leq c_{1} \; \sum_{k=0}^{\INF}\left | \beta_{k}(\xi) \right |
\leq c_{1} c_{2},
\end{align*} we obtain
\begin{align}\label{EstVinCalp}
\ds \sup_{ n \leq 0}  \sup_{\xi \in \ol\OM} |v_{-n}(\xi)| < \INF.
\end{align}

From \eqref{EstVinCalp} and by taking the supremum in $\xi,\eta\in\ol\OM$ with $\xi\neq\eta$, for each $n \leq 0$, in
\begin{align*}
 \frac{\lvert  \bv_{n}(\xi) - \bv_{n}(\eta) \rvert}{ \ds \lvert \xi - \eta \rvert^{\alpha}}
 \leq
\sum_{k=0}^{\INF} & \frac{ \ds \lvert \beta_{k}(\xi) - \beta_{k}(\eta) \rvert}{ \ds \lvert \xi - \eta \rvert^{\alpha}} | g_{n-k}(\xi)|  \\&+    \sum_{k=0}^{\INF}|\beta_{k}(\eta)|
\frac{ \ds \lvert g_{n-k}(\xi) - g_{n-k}(\eta) \rvert}{ \ds \lvert \xi - \eta \rvert^{\alpha}}, \\ \nn
& \leq
c_{1} \sum_{k=0}^{\INF} \frac{ \ds \lvert \beta_{k}(\xi) - \beta_{k}(\eta) \rvert}{ \ds \lvert \xi - \eta \rvert^{\alpha}} + \kappa_{1} \sup_{\eta \in \ol \OM} \sum_{k=0}^{\INF}|\beta_{k}(\eta)| \\ \nn
& \leq
c_{1} \kappa_{2} + c_{2} \kappa_{1},
\end{align*}
we obtain that $\bv \in C^{\alpha} (\ol\OM ; l_{\INF})$.

Next we prove part (ii). Let $\bg \in C^{\alpha} (\ol\OM ; l^{1}) $, and let $\bv = \Pplus (e^h)\ast_{n} \bg$ be as before.
Since $\bg ,\Pplus(e^h)\in C^\alpha(\ol\OM;l^1) $, we have
\begin{align}\label{Estgl1}
c_{3}:= \sup_{\xi \in \ol\OM} \sum_{n=0}^{\INF} |g_{-n}(\xi)| < \INF,
\kappa_{3}:= \underset{{\substack{
            \xi,\eta \in \ol\OM \\
            \xi \neq \eta } }}{\sup} \sum_{n=0}^{\INF}\, \frac{ \ds \lvert g_{-n}(\xi) - g_{-n}(\eta) \rvert}{ \ds \lvert \xi - \eta \rvert^{\alpha}}< \INF.
\end{align}
By taking the supremum in $\xi\in\overline\OM$ in
\begin{align*}
 \sum_{n=0}^{\INF}|v_{-n}(\xi)| \leq  \sum_{n=0}^{\INF} \sum_{k=0}^{\INF}|\beta_{k}(\xi)||g_{n-k}(\xi)| \leq \sum_{k=0}^{\INF}|\beta_k(\xi)| \sum_{n=0}^{\INF}|g_{-n-k}(\xi)|\\
 \leq c_{3} \sum_{k=0}^{\INF}\left |\beta_k(\xi) \right |
\leq c_{2} c_{3},
\end{align*}we obtain
\begin{align}\label{EstVinCalpl1}
\ds \sup_{\xi \in \ol\OM} \sum_{n=0}^{\INF}|v_{-n}(\xi)| < \INF.
\end{align}
From \eqref{EstVinCalpl1} and by taking the supremum in $\xi,\eta\in\ol\OM$ with $\xi\neq\eta$ in
\begin{align*}
\frac{ \| \bv(\xi) - \bv(\eta)\|_{l_1} }{| \xi - \eta |^\alpha}
& \leq
\sum_{k=0}^{\INF} \frac{|\beta_{k}(\xi) - \beta_{k}(\eta) |}{ |\xi - \eta|^\alpha} \sum_{n=0}^{\INF}| g_{-n-k}(\xi)|\\
            &\quad+ \sum_{k=0}^{\INF} |\beta_{k}(\eta)|\sum_{n=0}^{\INF}
\frac{|g_{-n-k}(\xi) - g_{-n-k}(\eta)|}{|\xi - \eta|^\alpha}\\
& \leq
c_{3} \kappa_{2} + c_{2} \kappa_{3},
\end{align*}
we obtain that $\bv \in C^{\alpha} (\ol\OM ; l^{1})$.

Last we prove part (iii).

 Since $a \in C_0^{1,\alpha}(\ol\OM)$, it follows from Proposition \ref{hprop} that $e^{h} \in C^{1,\alpha}(\Gam \times \sph) \subset C^{\alpha}(\Gam;C^{\alpha}(\sph))$,
 and from \eqref{betaLNorm} and \eqref{betaCalpNorm}, we have
\begin{align*}
c_{4}:= \sup_{\xi \in \Gam} \sum_{k=0}^{\INF}|\beta_{k}(\xi)| < \INF,\;
\kappa_{4}:= \underset{{\substack{
            \xi,\eta \in \Gam \\
            \xi \neq \eta } }}{\sup} \sum_{k=0}^{\INF}\, \frac{ \ds \lvert \beta_{k}(\xi) - \beta_{k}(\eta) \rvert}{ \ds \lvert \xi - \eta \rvert^{\alpha}}< \INF.
\end{align*}

Let $\bg \in Y_{\alpha},$ and let $\bv = \Pplus (e^h)\ast_{n} \bg$ be as before.

Since $\bg \in Y_{\alpha}$, we have
\begin{align*}
c_{5}&:= \sup_{\xi \in \Gam} \sum_{j=1}^{\INF}  j^2 \lvert g_{-j}(w) \rvert < \INF, \;
\kappa_{5}&:= \underset{{\substack{
            \xi, \mu  \in \Gam \\
            \xi \neq \mu } }}{\sup} \sum_{j=1}^{\INF} j\frac{\lvert g_{-j}(\xi) - g_{-j}(\mu) \rvert}{\lvert \xi - \mu \rvert^{\alpha}} < \INF.
\end{align*}

By taking the supremum in $w \in \Gam$ in
\begin{align*}
\sum_{j=1}^{\INF}j^2 \lvert v_{-j}(w) \rvert
& \leq  \sum_{j=1}^{\INF} j^{2}\sum_{k=0}^{\INF} |\beta_{k}(w)| \lvert g_{-j-k}(w) \rvert  \\
& \leq \sum_{k=0}^{\INF} |\beta_{k}(w)| \sum_{j=1}^{\INF} j^{2}\lvert g_{-j-k}(w) \rvert  \\
& \leq \sum_{k=0}^{\INF} |\beta_{k}(w)| \sum_{j=1}^{\INF} j^{2}\lvert g_{-j}(w) \rvert  \\
& \leq c_{4} c_{5},
\end{align*} we obtain that $\bv \in l^{1,2}_{\INF}(\Gamma)$.

Finally we show that $\bv $ obeys the estimate in \eqref{YGamdefn}.
By taking the supremum in $\xi,\mu\in\Gam$ with $\xi\neq\mu$ in
\begin{align*}
\sum_{j=1}^{\INF}&
            \frac{j\lvert v_{-j}(\xi) - v_{-j}(\mu) \rvert}{\lvert \xi - \mu \rvert^\alpha}
\\
&\leq
\sum_{j=1}^{\INF} \frac{j}{\lvert \xi - \mu \rvert^\alpha}
  \sum_{k=0}^{\INF} \lvert \beta_{k}(\xi)\;g_{-j-k}(\xi) - \beta_{k}(\mu)\;g_{-j-k}(\mu) \rvert \\
&\leq
\sum_{j=1}^{\INF} j
  \sum_{k=0}^{\INF} \frac{\lvert \beta_{k}(\xi) - \beta_{k}(\mu)\rvert}{\lvert \xi - \mu \rvert^\alpha} |g_{-j-k}(\xi)| \\
&\qquad+
\sum_{j=1}^{\INF} j
            \sum_{k=0}^{\INF} \frac{\lvert g_{-j-k}(\xi) - g_{-j-k}(\mu)\rvert}{\lvert \xi - \mu \rvert^\alpha} |\beta_{k}(\mu)|   \\
            &\leq
\sum_{k=0}^{\INF}  \frac{\lvert \beta_{k}(\xi) - \beta_{k}(\mu)\rvert}{\lvert \xi - \mu \rvert^\alpha}
  \sum_{j=1}^{\INF}j |g_{-j}(\xi)| \\
&\qquad+ \sum_{k=0}^{\INF} |\beta_{k}(\mu)|
\sum_{j=1}^{\INF} \frac{j\lvert g_{-j}(\xi) - g_{-j}(\mu)\rvert}{\lvert \xi - \mu \rvert^\alpha}  \\
&\leq
\kappa_{4} c_{5} + c_{4} \; \kappa_{5},
\end{align*} we obtain that $\bv \in Y_{\alpha}$.
\end{proof}

Recall the Hilbert transform  $\HT_0$ in Definition \ref{hilbertT_definition}, $\mathcal{P}_{\pm}$ in \eqref{Pminusplus}, and $e^{\pm h}$ in \eqref{ehEq}.
\begin{definition}\label{hilbertAtt_definition}
The Hilbert transform associated with the attenuated Radon transform for $g\in C^{1,\alpha}(\Gam \times \sph)$  is given by
\begin{align}\label{hilbertAT}
 \HAT (\Pminus (g)) :=  \Pplus(e^h) \ast_{n}  \HT_0 \left ( \Pplus(e^{-h}) \ast_{n} \Pminus  (g) \right)
\end{align}
where $\ast_{n}$ is the convolution operator on sequences.
\end{definition}

Using the Fourier coefficients of $e^{\pm h}$, we can also write for $\bu: = \langle u_{0},u_{-1},u_{-2},... \rangle$, the Hilbert transform as
\begin{align*}\label{hilbertATAlt}
 \HAT \bu :=   \sum_{m=0}^{\INF} \beta_{m} \mathcal{L}^{m} \left ( \HT_0 \left ( \sum_{k=0}^{\INF} \alpha_{k} \mathcal{L} ^{k} \right ) \right ) \bu
\end{align*}
where $\mathcal{L}$ is the left translation operator and  $\alpha_{k}, \beta_{k}$ are the Fourier coefficients of $e^{-h(x,\tta)}$, respectively, $e^{h(x,\tta)}$ as in \eqref{ehEq}.

The following result describes the mapping properties of the Hilbert transform $\HAT$ needed later.
\begin{prop}\label{HAttproperties}
Let $l^{1,1}_{\INF}(\Gam)$ and $C^{\ds \eps}(\Gamma ; l_1)$ be the spaces in \eqref{lGamdefn} and \eqref{CepsGamdefn} respectively.
Assume $a \in C_{0}^{1,\alpha}(\ol\OM)$ with $\alpha>1/2$, and $\eps>0$ be arbitrarily small. Then
\begin{equation}
\HAT :  C^{\eps}(\Gam ; l_{1})\cap l^{1,1}_{\INF}(\Gam) \longrightarrow C^{ \eps}\left( \Gamma; l_{\INF} \right).
\end{equation}

\end{prop}

\begin{proof}
Let $g \in C^{1,\alpha}(\Gam \times \sph) \subset C^{ \eps}(\Gam ; C^{1,\alpha}(\sph)) $, then by Proposition \ref{gCepsgbarXY}(i), $\Pminus g \in l^{1,1}_{\INF}(\Gam)\cap C^\eps(\Gamma;l^1)$. Since $a \in C_0^{1,\alpha}(\ol\OM)$, it follows from Proposition \ref{hprop} that $e^{\pm h} \in C^{\ds \eps}(\Gam ; C^{1,\alpha}(\sph))$.

Since $e^{-h} g \in C^{\ds \eps}(\Gam ; C^{1,\alpha}(\sph))$, it follows from
Proposition \ref{gCepsgbarXY} (i) that $\Pminus (e^{-h}g) \in l^{1,1}_{\INF}(\Gam)\cap C^\eps(\Gamma;l^1)$.
By \eqref{PmpsProp}, $\Pminus (e^{-h}g) = (\Pplus e^{-h}) \ast_{n} (\Pminus (g)) $ and so by Proposition \ref{Hproperties}, $ \HT_0 \left ( \Pplus(e^h) \ast_{n} \Pminus (g) \right) \in C^{\ds \eps}(\Gam ; l_{\INF})$.
Finally by Proposition \ref{ehConvProp} (ii), $\Pplus(e^h) \ast_{n}  \HT_0 \left ( \Pplus(e^h) \ast_{n} \Pminus  (g) \right) \in C^{\ds \eps} (\Gam ; l_{\INF}) $.
\end{proof}

Now we are able to state and prove our main result.

\begin{theorem}[Range characterization for the Attenuated Radon transform]\label{AttRTTh}

Let $\OM\subset{R}^2$ be a domain with $C^2$ boundary $\Gamma$ of strictly positive curvature, and $a \in C_0^{1,\alpha}(\ol\OM)$, $\alpha>1/2$ be real valued.

(i) Let $f \in C_0^{1,\alpha}(\ol\OM)$ be real valued. Then $R_{a}f\cap C^{\alpha}(\Gam;C^{1,\alpha}(\sph)) \neq \emptyset$, and if
$g \in R_{a}f\cap C^{\alpha}(\Gam;C^{1,\alpha}(\sph))$, its projection $\Pminus (g)$ must solve
\begin{align}\label{AttRTThCond}
[I+i\HAT]\Pminus (g)=0,
\end{align}with the Hilbert transform $\HAT$ defined in \eqref{hilbertAT}.

(ii) Let $g\in C^{\alpha} \left(\Gam; C^{1,\alpha}(\sph) \right)\cap C^0(\Gam;C^{2,\alpha}(\sph))$ be real valued with the projection 
$\Pminus (g)$ satisfying \eqref{AttRTThCond}. Then there exists a real valued $f \in C^\alpha(\OM)\cap L^1(\OM)$ for which $g\in R_af$.
\end{theorem}

\begin{proof} (i) For $z\in\OM$ and $\theta\in S^1$, let $u(z,\tta)$ be the solution of
\begin{align}\label{advectUGamMinus}
    \tta\cdot\nabla u(z,\tta) +a(z) u(z,\theta) &= f(z), \quad  (z,\tta)\in \OM \times S^{1}, \\ \nn
 u(z,\tta) &= 0, \quad  (z,\tta) \in \Gam_{-},
\end{align} namely $u(z+t\tta,\tta) = \ds \int_{0}^{t}f(z+s\theta)\, e^{-Da(z+s\theta,\theta)}ds$, for $(z,\tta)\in \Gam_{-}$ and $0 \leq t \leq \tau_{+}(z,\theta)$, where $\Gam_{\pm} = \left \{ (z,\tta) \in \Gam \times \sph : \pm \bn(z)\cdot \tta > 0\right \}$ denote the incoming(-), respectively, outgoing (+) boundary and $n(z)$ denotes the outer normal at some boundary point $z$.

Let $g(z,\tta) := \ds u(z,\tta) \lvert_{\Gam \times \sph}$.
Note that $\Gam \times \sph = \Gam_{-} \cup \Gam_{+} \cup Z $, where $Z$ is the variety in \eqref{variety}.
Since $g(z,\tta) = 0$ for $(z,\tta) \in \Gam_{-} \cup Z$ and $g(z,\tta) = \ds \int_{0}^{\tau_{+}(z,\theta)}f(z+s\theta)\, e^{-Da(z+s\theta,\theta)}ds$, for  $(z,\tta) \in \Gam_{+}$, it follows that $g$ satisfies \eqref{Radon_definition} and thus $ g \in R_{a}f$.

Since $a \in C_0^{1,\alpha}(\ol\OM)$, it follows from Proposition \ref{hprop} that $e^{-Da} \in C^{1,\alpha}(\ol \OM\times\sph)$ and so $fe^{-Da} \in C^{1,\alpha}_{0}(\ol \OM\times\sph) \subset C^{ \alpha}(\ol \OM;C^{1,\alpha}(\sph)) $. The proof of Corollary \ref{nonempty} applied to $fe^{-Da}$ shows that $g \in C^{1,\alpha}(\Gam\times\sph) $ and therefore $ g \in R_{a}f\cap C^{ \alpha}(\Gam;C^{1,\alpha}(\sph))$.

For $z\in \ol \OM$ and $\theta\in S^1$, if we let
\begin{align}  \label{vNuEq}
v(z,\tta) := e^{-h(z,\tta)}u(z,\tta),
\end{align}
where $u(z,\tta)$ solves \eqref{advectUGamMinus} with
$ \ds u(z,\tta) \lvert_{\Gam \times \sph} = g(z,\tta) $, and $e^{-h(z,\tta)}$ as in \eqref{ehEq} then $v(z,\tta)$ solves
\begin{align}\label{advectionVwithg}
 \tta\cdot\nabla v(z,\tta)  &= f(z)  e^{-h(z,\tta)} \quad (z,\tta) \in \OM \times \sph \\ \nn
  v \lvert_{\Gam \times \sph} &= g \; e^{-h} \lvert_{\Gam \times \sph}
\end{align}
If $\bv:=\langle v_0,v_{-1},v_{-2},...\rangle$ is the projection on the non-positive Fourier coeficients of $\ds \sum_{n=-\INF}^{\INF} v_{n}(x) e^{in\fii}$ then the equation \eqref{advectionVwithg} yields for each $n=0,-1,-2,...$
\begin{align*}
\ol{\del} v_{n}(z) + \del v_{n-2}(z) = 0, \quad z \in \OM.
\end{align*}This makes $\bv:=\langle v_0,v_{-1},v_{-2},...\rangle$ be A-analytic.

The convolution applied to \eqref{vNuEq} rewrites $\bv$ as
\begin{align}\label{BVinBU}
\bv(z) = \Pplus (e^{-h(z,\tta)}) \ast_{n} \Pminus (u(z,\tta)), \quad (z,\tta) \in \OM \times \sph.
\end{align}

Since $a \in C_0^{1,\alpha}(\ol\OM)$ and $ g \in C^{ \alpha}(\Gam;C^{1,\alpha}(\sph))$, we have from Proposition \ref{hprop}, $e^{-h} g \in C^{\alpha}(\Gam;C^{1,\alpha}(\sph))$. Hence, by Proposition \ref{gCepsgbarXY} (i), $\Pminus( e^{-h} g ) \in l^{1,1}_{\INF}(\Gam)\cap C^{\alpha}(\Gamma,l^1)$.

Since $\Pminus( e^{-h}g) $ is the boundary value of the A-analytic function $\bv$, we can apply necessity part in Theorem \ref{NecSuf} to conclude that
\begin{align}\label{RTThCond}
(I+i\HT_0)\Pminus (g \; e^{-h} \lvert_{\Gam \times \sph} )=0.
\end{align}
The convolution of \eqref{RTThCond} by $\Pplus (e^h)$ yields
\begin{align*}
0 &= \Pplus (e^h) \ast_{n} (I+i\HT_0)\Pminus (e^{-h}g), \\ \nn
&=\Pplus (e^h) \ast_{n} \Pminus (e^{-h}g)  +i \Pplus (e^h) \ast_{n} \HT_0\Pminus (e^{-h}g), \\ \nn
&=\Pminus (g) + i \HAT \Pminus (g), \\ \nn
&= [I+i\HAT] \Pminus (g) .
\end{align*}
In the third equality above we use \eqref{PmpsProp} to simplify $$\Pplus (e^h) \ast_{n} \Pminus (e^{-h}g) = \Pminus (e^{h}e^{-h}g) = \Pminus (g),$$ and Definition \ref{hilbertAtt_definition} of $\HAT$ to obtain $$\Pplus(e^h) \ast_{n}  \HT_0 \left ( \Pplus(e^h) \ast_{n} \Pminus  (g) \right) = \HAT \Pminus (g).$$

Conversely, let $g\in C^{\alpha} \left(\Gam; C^{1,\alpha}(\sph) \right)\cap C^0(\Gam;C^{2,\alpha}(\sph))$ be real valued and such that $\Pminus (g)$ satisfies \eqref{AttRTThCond}. Then by Proposition \ref{gCepsgbarXY} (ii), we have $\Pminus (g)\in Y_{\alpha}$.
Since $a \in C_0^{1,\alpha}(\ol\OM)$, it follows from Propositions \ref{hprop} and \ref{gCepsgbarXY}(i), that $\Pplus (e^{h}) \in l^{1,1}_{\INF}(\Gam)\cap C^\alpha(\Gamma;l^1)$.
Finally we apply Proposition \ref{ehConvProp}(iv) to yield
$\Pminus( e^{-h} g ) \in Y_{\alpha}$.
From $\Pminus (g)$ satisfying \eqref{AttRTThCond} we have
\begin{align}\label{ARTThCond}
0 =[I+i\HAT] \Pminus (g) =\Pminus (g) + i \HAT \Pminus (g).
\end{align}
The convolution of \eqref{ARTThCond} by $\Pplus (e^{-h})$ yields
\begin{align*}
0 &=\Pplus (e^{-h}) \ast_{n} \left ( \Pminus (g) + i \HAT \Pminus (g) \right ), \\ \nn
&=\Pplus (e^{-h}) \ast_{n} \Pminus (g)  + i \Pplus (e^{-h}) \ast_{n} \HAT \Pminus (g), \\ \nn
&=  \Pminus (e^{-h}g) +i \Pplus (e^{-h}) \ast_{n} \Pplus (e^{h}) \ast_{n} \HT_0\Pminus (e^{-h}g), \\ \nn
&=  \Pminus (e^{-h}g) +i \Pplus (1) \ast_{n} \HT_0\Pminus (e^{-h}g), \\ \nn
&=  \Pminus (e^{-h}g) +i \HT_0\Pminus (e^{-h}g), \\ \nn
&= [I+i\HT_0] \Pminus (e^{-h}g).
\end{align*}
In the third equality above we use the Proposition \ref{ehConvProp} part(iii), to simplify $\Pplus (e^{-h}) \ast_{n} \Pminus (g) = \Pminus (e^{-h}g) $, and Definition \ref{hilbertAtt_definition} of $\HAT$. In the fourth equality above we use $\Pplus (e^{-h}) \ast_{n} \Pplus (e^{h}) = \Pplus(1) := \langle 1,0,0,\cdots \rangle$, and the fact that
$\Pplus(1)$ is the identity element for convolution in sequences to conclude
$\Pplus (1) \ast_{n} \HT_0\Pminus (e^{-h}g) = \HT_0\Pminus (e^{-h}g)$.

For each $z \in \OM$, construct the vector valued function
$\bv = \langle v_0, v_{-1}, v_{-2}, ...\rangle $ by the Cauchy Integral formula \eqref{vnDefn}:
\[v_{n} (z)= \frac{1}{2}(G \bg)_{n}(z)+ (C\bg)_{n}(z), n=0,-1,-2...,\] where $\bg:= \Pminus (e^{-h}g)$.
By the Corollary \ref{CorTh2NecY}, $\bv \in C^{1,\eps}(\OM;l_1) \cap C^{\eps}(\ol \OM;l_1)$ is $A$-analytic and $\bv|_{\Gam}=\bg$.

Construct the vector valued function
$\bu: = \langle u_{0},u_{-1},u_{-2},... \rangle$ from $\bv$ by the convolution formula $\bu(z) = \Pplus (e^{h(z,\cdot)}) \ast_{n} \bv(z)$ for $(z,\cdot) \in \OM \times \sph$. By the Proposition \ref{ehConvProp}(ii) we have $\bu \in C^{\alpha}(\ol \OM; l_{1})$ and by Proposition \ref{gCepsgbarXY}(iii) we have $u(z,\tta) := \Pstar (\bu(z)) \in C^{1,\alpha}(\OM \times \sph) \cap C^{\alpha}(\ol \OM \times \sph)$. Note that
\begin{align*}
\Pminus (u \lvert_{\Gam \times \sph}) &=\Pplus (e^{h} \lvert_{\Gam \times \sph} ) \ast_{n} \bv \lvert_{\Gam}, \\ \nn
&=\Pplus (e^{h} \lvert_{\Gam \times \sph} ) \ast_{n} \Pminus (\ds e^{-h} \lvert_{\Gam \times \sph} \; g) , \\ \nn
&=  \Pminus (g)
\end{align*}
Taking $\Pstar$ on both sides of the above equation and using the fact that $u$ and $g$ are real valued yields
$u \lvert_{\Gam\times \sph} = g$.

We define the H\"{o}lder continuous function $f\in C^\alpha(\OM)$ by
\begin{equation}\label{fDefn}
f(z) := \tta\cdot\nabla u(z,\tta) +a(z) u(z,\tta) ,\quad (z,\tta)\in\OM \times \sph,
\end{equation}
and show that $f$ integrates along any line and that $g\in R_af$.

Since $e^{-Da}$ in \eqref{divbeam} is an integrating factor, the equation \eqref{fDefn} can be rewritten in the advection form as
\begin{align*}
\ds f(z)  e^{-Da(z,\tta)} = \tta\cdot\nabla \left( e^{-Da(z,\tta)} u(z,\tta) \right ) .
\end{align*} and integrated along lines in direction $\tta$ to obtain
\begin{align*}
\ds \int_{\tau_-(x,\theta)}^{\tau_+(x,\theta)} f(x+t\theta)e^{-Da(x+t\theta,\theta)}dt
&= \ds \left. e^{-Da(z+t\tta,\tta)} u(z+t\tta,\tta) \ds  \right
\lvert_{\tau_{-}(z,\theta)}^{\tau_{+}(z,\theta)} \\ \nn
&= \ds  e^{-Da(z^{+}_{\tta},\tta)} u(z^{+}_{\tta},\tta)  - e^{-Da(z^{-}_{\tta},\tta)} u(z^{-}_{\tta},\tta) \ds \\ \nn
&= g(x^+_\theta,\theta)-\left[e^{-Da} g\right](x^-_\theta,\theta),
\end{align*} where the notation $z^\pm_\theta=z\pm\tau_\pm(z,\theta)\theta$ as in \eqref{xthetapm}.
This shows that $f$ integrates along any arbitrary line, in particular $f\in L^1(\OM)$, and that $g\in R_{a}f$.
\end{proof}



\begin{thebibliography}{10}
\bibitem{ABK}
E.~V. Arzubov, A.~L. Bukhgeim and S.G. Kazantsev,
\textit{Two-dimensional tomography problems and the theory of
A-analytic functions}, Siberian Adv. Math. \textbf{8}(1998),
1--20.

\bibitem{bal} G. Bal, \textit{On the attenuated Radon transform with full and
partial measurements}, Inverse Problems \textbf{20}(2004),
399--418.

\bibitem{bomanStromberg} J. Boman and J.-O. Str\"omberg, \textit{Novikov's inversion formula
for the attenuated Radon transform--a new approach}, J. Geom.
Anal. \textbf{14}(2004), 185--198.


\bibitem{bukhgeim_book}
A.~L. Bukhgeim, \textit{\em Inversion Formulas in Inverse
Problems}, in Linear Operators and Ill-Posed Problems by M. M.
Lavrentev and L. Ya.  Savalev, Plenum, New York, 1995.

\bibitem{finch} D. V. Finch, \textit{The attenuated x-ray transform: recent
developments}, in Inside out: inverse problems and applications,
Math. Sci. Res. Inst. Publ., 47, Cambridge Univ. Press, Cambridge,
2003, 47--66.

\bibitem{gelfandGraev} I. M. Gelfand and M.I. Graev, \textit{Integrals over hyperplanes of basic and generalized functions}, Dokl. Akad. Nauk SSSR {\bf135} (1960), no.6, 1307-1310; English transl., Soviet Math. Dokl. {\bf 1} (1960), 1369-1372.

\bibitem{helgason} S. Helgason, \textit{An analogue of the Paley-Wiener theorem for the Fourier transform on certain symmetric spaces},
Math. Ann. {\bf 165} (1966), 297�308.


\bibitem{helgasonBook} S. Helgason, The Radon Transform, Birkh\"{a}user, Boston, 1999.

\bibitem{katznelson} Y. Katznelson, \textit{An introduction to harmonic
analysis}, Cambridge Math. Lib., Cambridge, 2004.



\bibitem{ludwig} D. Ludwig, \textit{The Radon transform on euclidean space}, Comm. Pure Appl. Math. {\bf19} (1966), 49�81.

\bibitem{muskhellishvili} N.I. Muskhelishvili, Singular Integral Equations, Dover, New York, 2008.

\bibitem{naterrerBook} F. Natterer, \textit{The mathematics of computerized
tomography}, Wiley, New York, 1986.

\bibitem{naterrer01} F. Natterer, \textit{Inversion of the Attenuated Radon
transform}, Inverse Problems \textbf{17}(2001), 113-119.

\bibitem{novikov01} R. G. Novikov, \textit{Une formule d'inversion pour la
transformation d'un rayonnement X att�nu�},  C. R. Acad. Sci.
Paris S�r. I Math., \textbf{332}(2001),  1059--1063.

\bibitem{novikov02} R. G. Novikov, \textit{On the range characterization for the two-dimensional attenuated x-ray transformation},
Inverse Problems {\bf 18} (2002), no. 3, 677�700.

\bibitem{paternainUhlmann} G. P. Paternain, M. Salo, and G. Uhlmann, \textit{On the attenuated Ray transform for unitary connections}, arXiv:1302.4880v1, 2013.

\bibitem{pestovUhlmann} L. Pestov and G. Uhlmann, \textit{On characterization of the range and inversion formulas for the geodesic X-ray transform}, Int. Math. Res. Not. {\bf 80} (2004), 4331�4347.



\bibitem{tamasan02}
A.~Tamasan, \textit{An inverse boundary value problem in two-dimensional transport}, Inverse Problems \textbf{18}(2002), 209--219.

\bibitem{tamasan03}
A.~Tamasan, \textit{Optical tomography in weakly anisotropic
scattering media}, Contemporary Mathematics \textbf{333}(2003),
199--207.
\bibitem{tamasan07} A. Tamasan, \textit{Tomographic reconstruction of vector fields in variable background media} Inverse Problems \textbf{23}(2007), 2197--2205.
\end{thebibliography}
\end{document}